\documentclass{amsart}

\usepackage{amsmath}
\usepackage{amssymb}
\usepackage{amsthm}
\usepackage{amsfonts}
\usepackage{dsfont}
\usepackage{color}
\usepackage{enumerate}
\usepackage[margin=38mm]{geometry}
\usepackage{graphicx}
\usepackage{braket}
\usepackage{tensor}

\newcommand{\reals}{\mathbb{R}}
\newcommand{\rationals}{\mathbb{Q}}
\newcommand{\complex}{\mathbb{C}}

\newcommand{\integers}{\mathbb{Z}}

\newcommand{\bracketc}[1]{\bigg[#1\bigg]}

\newcommand{\para}[1]{\left(#1\right)}
\newcommand{\paraa}[1]{\big(#1\big)}
\newcommand{\parab}[1]{\Big(#1\Big)}
\newcommand{\parac}[1]{\bigg(#1\bigg)}

\newcommand{\diag}{\operatorname{diag}}

\newcommand{\spacearound}[1]{\quad#1\quad}
\newcommand{\equivalent}{\spacearound{\Leftrightarrow}}
\renewcommand{\implies}{\spacearound{\Rightarrow}}

\newtheorem{theorem}{Theorem}[section]

\newtheorem{lemma}[theorem]{Lemma}
\newtheorem{proposition}[theorem]{Proposition}
\newtheorem{example}[theorem]{Example}
\theoremstyle{definition}
\newtheorem{definition}[theorem]{Definition}
\theoremstyle{remark}

\numberwithin{equation}{section}

\renewcommand{\emph}[1]{\textit{#1}}
\newcommand{\F}{\mathcal{F}}
\newcommand{\Fall}{F(\reals^D,\complex)}
\newcommand{\A}{\mathcal{A}}

\newcommand{\Ahq}{\A_{\hbar,q}}
\newcommand{\AhqD}{\A_{\hbar,q}^D}
\newcommand{\AhqDF}{\A_{\hbar,q}^D(\F)}

\renewcommand{\mid}{\mathds{1}}

\newcommand{\mathand}{\qquad\text{and}\qquad}
\newcommand{\qand}{\quad\text{and}\quad}

\newcommand{\Aut}{\operatorname{Aut}}
\newcommand{\Sh}{S_{\hbar}}
\newcommand{\Ssh}{\S_{\hbar}}
\newcommand{\ZD}{\integers^D}
\renewcommand{\S}{\mathcal{S}}

\newcommand{\Mat}{\operatorname{Mat}}
\newcommand{\Matbar}{\overline{\Mat}}

\newcommand{\Lt}{\tilde{\Lambda}}
\newcommand{\Et}{\tilde{E}}
\newcommand{\Rt}{\tilde{R}}
\newcommand{\Mt}{\tilde{M}}
\newcommand{\GL}{\operatorname{GL}}

\newcommand{\Shd}{\S^{\delta}_{\hbar}}

\newcommand{\deltat}{\tilde{\delta}}

\newcommand{\g}{\mathfrak{g}}

\title[On the construction of fuzzy spaces and modules over shift algebras]{On the construction of fuzzy spaces and\\ modules over shift algebras}

\author{Joakim Arnlind}
\address[Joakim Arnlind]{Dept. of Math.\\
Link\"oping University\\
581 83 Link\"oping\\
Sweden}
\email{joakim.arnlind@liu.se}

\author{Andreas Sykora}
\address[Andreas Sykora]{}
\email{syko@gelbes-sofa.de}

\thanks{}

\subjclass[2000]{}

\begin{document}

\begin{abstract}
  We introduce shift algebras as certain crossed product algebras
  based on general function spaces and study properties, as well as
  the classification, of a particular class of modules depending on a
  set of matrix parameters. It turns out that the structure of these
  modules depends in a crucial way on the properties of the function
  spaces. Moreover, for a class of subalgebras related to compact
  manifolds, we provide a construction procedure for the corresponding
  fuzzy spaces, i.e. sequences of finite dimensional modules of
  increasing dimension as the deformation parameter tends to zero, as
  well as infinite dimensional modules related to fuzzy non-compact
  spaces.
\end{abstract}

\maketitle
\tableofcontents

\section{Introduction}

\noindent
Over the past decades, noncommutative geometry has emerged as a
crucial ingredient for physical theories describing the unification of
quantum mechanics and general relativity. At small length scales, or
high energies, there are good reasons to believe that space itself
becomes noncommutative (see
e.g. \cite{dfr:quantum.structure.spacetime}). For instance, through the
framework of spectral triples (\cite{c:ncgbook}) the standard
model of particles has been formulated in terms of noncommutative
geometry as the Spectral Standard Model
\cite{cc:spectral.action.principle}.

In parallel, the theory of ``Fuzzy spaces'' grew out of the quantization
of Membrane theory: a theory of quantum gravity built on the idea of
describing fundamental particles as two dimensional surfaces, rather
than one dimensional objects as in String theory. To quantize the
theory, one uses a regularization procedure where functions are
replaced by sequences of matrices (of increasing dimension), enabling
a straight-forward quantization (cf. \cite{h:phdthesis}).

Parallel to this, in string theory the so-called IKKT matrix model emerged,
which can be derived via a matrix regularization of the worldsheet action
functional \cite{ikkt:superstring} or can be seen as a compactification
of ten-dimensional super Yang-Mills theory to a point. Motivated by the IKKT
model and other matrix models, more complicated fuzzy spaces,
also in higher dimensions, were studied and used as physical models,
such as emergent gravity (\cite{s:emergent.intro, Steinacker_2014, 
Carow-Watamura:2004,Ramgoolam_2002,ss:covariant.spheres, Sperling_2019}). 

Hence, it becomes important to understand the geometry of such matrix
regularizations, as well as finding explicit examples. Although there
are general existence results (see e.g. \cite{bms:toeplitz}), showing
that one may find matrix regularizations for arbitrary (quantizable)
compact K\"ahler manifolds, it does not provide a deeper understanding
on how geometrical and topological properties of the manifold are
reflected in the regularization. Moreover, for a long time, the fuzzy
sphere and the fuzzy torus were more or less the only known explicit
examples.  In \cite{abhhs:noncommutative}, a one parameter class of
surfaces was considered, interpolating between spheres and tori
(including the singular point of topology change). For the first time,
this allowed for the explicit study of how geometrical deformation and
topology change affect the matrix regularization. It turned out the
topology change implied a drop in the dimension of the matrix
representation. By now there are many more examples of fuzzy spaces,
as well as a broader understanding of their geometry (see
e.g. \cite{s:membrane.topology,aht:spinning,s:emergent.intro,ahh:multilinear,s:fuzzy.construction.kit,ss:covariant.spheres,a:low.dimensional}. 

In many cases, matrix regularizations arise as finite dimensional
representations of a (noncommutative) algebra representing the
noncommutative space. For instance, this is the case in
\cite{abhhs:noncommutative}, where representations of an algebra
defined by cubic relations in the generators were considered. Fuzzy
spaces with a dimension higher than two can be constructed from
representations of finite dimensional Lie algebras, which are
interpreted as fuzzy homogeneous spaces.  For example, a fuzzy
4-hyperboloid is considered in \cite{Sperling_2019} and a "squashed"
fuzzy hyperboloid is interpreted as space time in
\cite{Steinacker_2018}. One of the main motivations for this paper is to
provide a simple way of constructing representations of fuzzy spaces
with a dimension higher than two.

In this paper, we study a class of \emph{shift algebras} defined as
twisted crossed product algebras via the action of $\integers^D$ on a
function algebra $\F$ consisting of complex valued functions on
$\reals^D$. If $\F$ is chosen to be the algebra of continuous functions,
then one recovers a crossed product algebra related to the
noncommutative cylinder (cf. \cite{vs:nc.cylinder}) along with a
particular choice of cocycle, in analogy with the algebra
constructed in \cite{al:projectors.cylinder}. However, for other
choices of $\F$, one obtains fundamentally different algebras such as
the noncommutative torus, when $\F$ is chosen to be the algebra of
constant functions.  In particular, we are interested in classes of
modules over shift algebras and a construction giving rise to fuzzy
spaces; i.e. sequences of finite dimensional modules of increasing
dimension as the deformation parameter $\hbar$ tends to zero.

The paper is organized as follows: In Section~\ref{sec:shift.algebras}
we define shift algebras and present a few results on isomorphisms for
different choices of
parameters. Section~\ref{sec:modules.shift.algebras} introduces a
class of modules, depending on a set of (matrix) parameters,
generalizing the construction in \cite{al:projectors.cylinder}, and
studies isomorphism classes with respect to different choices of
parameters. Moreover, for a subclass of modules, we classify all
simple modules in the case when the function algebra separates points.
In Section~\ref{sec:rep.subalgebras} we study modules over a general
class of subalgebras and present a construction procedure to generate
fuzzy spaces related to compact manifolds as well as non-compact manifolds.

In Section \ref{sec:shift.subalgebras.one.dimension} we consider
one-dimensional shift subalgebras and motivate that these algebras
relate to two-dimensional fuzzy spaces. We explicitly show that the
fuzzy sphere, a fuzzy catenoid and the fuzzy plane can be reproduced
with one-dimensional shift subalgebras. Section
\ref{sec:shift.subalgebras.two.dimensions} provides examples
illustrating the construction of fuzzy spaces with higher dimensional
shift subalgebras. In particular, we construct shift subalgebras which
can be interpreted as compact and non-compact four-dimensional level
sets immersed into $\reals^6$.

Section \ref{sec:shift.subalgebras.higher.dimensions} is devoted to
higher dimensional shift subalgebras related to Lie algebras. We show
that any Lie algebra with a finite dimensional representation can be
used to define elements in a shift subalgebra satisfying the
commutation relations of the Lie algebra.  The corresponding modules
can be interpreted as fuzzy homogeneous spaces.

\section{Shift algebras}\label{sec:shift.algebras}

\noindent
Let us recall some basic facts about crossed product algebras. Thus,
let $\mathcal{A}$ be a $\ast$-algebra and let $G$ be a discrete group
(of either finite or countably infinite cardinality). Given a group
action $\alpha: G \rightarrow\Aut(\mathcal{A})$, one can form the
crossed product algebra $\mathcal{A}\rtimes_\alpha G$. Moreover, given
a normalized $2$-cocyle $\omega:G\times G\to U(1)$, satisfying
\begin{align}\label{eq:cocycle.def} 
  &\omega_{g,h}\omega_{gh,k}=\omega_{g,hk}\omega_{h,k}\qquad
    \omega_{g,e}=\omega_{e,g} = 1
\end{align}
for $g,h,k\in G$, one can form the twisted crossed product algebra
$\A\rtimes_{\alpha,\omega} G$.  As a vector space, the twisted crossed product
algebra $\mathcal{A}\rtimes_{\alpha,\omega} G$ is defined as the set
of functions $G \rightarrow \mathcal{A}$ with compact support. When
the group is discrete, one may identify this set with the group ring
$\A[G]$ and every element $a\in\A\rtimes_{\alpha,\omega} G$ is written
as a formal (finite) linear combination
\begin{align}\label{eq:crossed.prod.element}
  a=\sum_{g\in G} a_gg
\end{align}
with $a_g\in\A$.  Multiplication is then defined by
\begin{align}\label{cpa_prod}
  a\cdot b &= \parab{ \sum_{g\in G} a_g g }\cdot \parab{ \sum_{h\in G} b_h h } =
  \sum_{g, h\in G} a_g \alpha_g(b_h) \omega_{g,h} gh
\end{align}
which may be written as
\begin{align*}
  a\cdot b &=\sum_{g\in G}\bracketc{\sum_{h\in G}a_h\alpha_h\paraa{b_{h^{-1}g}}\omega_{h,h^{-1}g}}g
  = \sum_{g\in G}\bracketc{\sum_{h\in G}a_{gh^{-1}}\alpha_{gh^{-1}}\paraa{b_{h}}\omega_{gh^{-1},h}}g.
\end{align*}
Furthermore, the algebra $\mathcal{A}\rtimes_{\alpha,\omega} G$ has an
involution defined by
\begin{align*} 
  a^\ast=\parab{ \sum_{g\in G} a_g g }^\ast
  = \sum_{g\in G}(\omega_{g^{-1},g})^{-1} \alpha_{g}(a_{g^{-1}}^\ast) g.
\end{align*} 
In the following, we will be interested in twisted crossed product
algebras where $\A$ is a subalgebra $\F$ of the $\ast$-algebra $\Fall$
of complex valued functions on $\reals^D$, and the group $G$ is chosen
to be $\ZD$. For $f\in\Fall$ and $\hbar>0$, define
$S_{\hbar,k}:\Fall\to\Fall$ as
\begin{align*}
  (S_{\hbar,k}f)(u_1,\ldots,u_D) = f(u_1,\ldots,u_k+\hbar,\ldots,u_D),
\end{align*}
and, for $\lambda\in\reals$, set
\begin{align*}
  (T_\lambda f)(u_1,\ldots,u_D)=f(\lambda u_1,\ldots,\lambda u_D).  
\end{align*}

\noindent 
The operators $S_{\hbar,k}$ allow one to define a group action
$\Sh:\ZD\to \Aut(\Fall)$ as
\begin{align*}
  \paraa{\Sh(k)f}(u_1,\ldots,u_D)
  &= \paraa{S^{k_1}_{\hbar,1}S^{k_2}_{\hbar,2}\cdots S^{k_D}_{\hbar,D}f}
    (u_1,\ldots,u_D)\\
  &=f(u_1+k_1\hbar,u_2+k_2\hbar,\ldots,u_D+k_D\hbar)
\end{align*}
for $k=(k_1,\ldots,k_D)\in\ZD$, and it is easy to check that $\Sh(k)$
is a $\ast$-automorphism fulfilling $\Sh(k)\Sh(l)=\Sh(k+l)$ for
$k,l\in\ZD$. For notational convenience we will write
$\Sh^k\equiv\Sh(k)$. Let us now introduce a 2-cocycle on $\ZD$.  For
$k=(k_1,\ldots,k_D),l=(l_1,\ldots,l_D)\in\ZD$ set
\begin{align}\label{eq:N.cocycle.def}
  N(k,l)= \sum_{m=2}^D\sum_{n=1}^{m-1}k_m l_n,
\end{align}
with the convention that $N(k,l)=0$ for $D=1$, and for arbitrary
$q\in\complex$ with $|q|=1$ define
\begin{align*}
  \omega(k,l) = q^{N(k,l)}.
\end{align*}
Since $N(k,l)$ is linear in both arguments, it follows immediately
that $\omega$ fulfills the 2-cocycle condition
\eqref{eq:cocycle.def}. For later convenience, we extend the
definition of $N$ to $\reals^D\times\reals^D$ using
\eqref{eq:N.cocycle.def} for $k,l\in\reals^D$.

\begin{definition}
  A $\ast$-subalgebra $\F\subseteq\Fall$ is called
  \emph{$\hbar$-invariant} if $S_{\hbar}^kf\in\F$ for all
  $k\in\integers^D$ and $f\in\F$.
\end{definition}

\noindent
In the following definition, we introduce type of algebras that will
be studied throughout the paper.

\begin{definition}\label{def:shift_alg}
  Let $\hbar>0$, $q\in\complex$ with $|q|=1$, and let $\F$ be a
  $\hbar$-invariant $\ast$-subalgebra of $\Fall$. The
  \emph{shift algebra $\AhqDF$} is defined as
  \begin{align*}
    \AhqDF = \F \rtimes_{S_\hbar,\omega} \ZD.
  \end{align*}
\end{definition}

\noindent
Note that when the function space $\F$ consists of continuous
functions, the construction above is similar to that of the
noncommutative cylinders found in \cite{vs:nc.cylinder} and
\cite{al:projectors.cylinder}. However, due to the freedom in choosing
the function algebra $\F$, Definition~\ref{def:shift_alg} also contains
compact manifolds, such as the noncommutative torus (when $\F$ is the
algebra of constant functions).

Following the notation in \eqref{eq:crossed.prod.element}, we 
write a generic element $f\in\AhqDF$ as
\begin{align*}
  f = \sum_{k\in\ZD}f_kU^k
  =\hspace{-4mm}\sum_{\,\,(k_1,\ldots,k_D)\in\ZD} f_{k_1\cdots k_D}
  U_1^{k_1}\cdots U_D^{k_D}
\end{align*}
with $U_1,\ldots,U_D$ being (multiplicative) generators of $\ZD$. To
simplify the notation, we shall often omit the explicit summation
symbol and assume that any repeated index is implicitly summed over
unless otherwise stated.  The algebra product is given as
\begin{align*}
  \parac{\sum_{k\in\ZD}f_kU^k}\parac{\sum_{l\in\ZD}g_lU^l}
  &= \sum_{k,l\in\ZD}q^{N(k,l)}f_k\paraa{\Sh^kg_l}U^{k+l}
  =\!\!\!\!\!\sum_{k,n\in\integers^D}q^{N(k,n-k)}f_k(\Sh^kg_{n-k})U^n.
\end{align*}
Note that in the case when $u_i\in \F$ (for $i=1,2,\ldots,D$) the
above product is induced by the relations
\begin{align*}
  &U_jU_i=qU_iU_j\quad\text{ (for $j>i$)}
  & &U_iu_j=u_jU_i\quad\text{ (for $i\neq j$)}\\
  &u_iu_j = u_ju_i& &U_iu_i = (u_i+\hbar)U_i
\end{align*}
for $i,j=1,2,\ldots,D$, perhaps providing a more intuitive
understanding of the product and the terminology \emph{shift algebra}. The
involution is computed as
\begin{align} \label{alg_involution}
  \parab{\sum_{k\in\ZD}f_kU^k}^\ast=
  \sum_{k\in\ZD}q^{N(k,k)}\paraa{\Sh^{k}\bar{f}_{-k}}U^{k},
\end{align}
where the bar denotes complex conjugation, implying that
$U_i^\ast=U_i^{-1}$ and $f^\ast=\bar{f}$.  Moreover, for convenience,
we introduce $u=(u_1,\ldots,u_D)$ and for $f\in\Fall$ we write
$f(u)=f(u_1,\ldots,u_D)$.

The shift algebra $\AhqDF$ depends on a choice of the parameters
$\hbar$ and $q$. In the following we show that if the function algebra
$\F$ has certain properties, then the shift algebras for different
choices of parameters are isomorphic. First, let us show that if $\F$
is scale invariant then the value of $\hbar$ is irrelevant.

\begin{proposition}\label{prop:Ahq.h.isomorphic}
  Let $\A_{\hbar_1,q}^D(\F)$ and $\A_{\hbar_2,q}^D(\F)$ be shift
  algebras. If $T_\lambda f\in\F$ for all $f\in\F$ and
  $\lambda\in\reals$ then $\A_{\hbar_1,q}^D(\F)\simeq\A_{\hbar_2,q}^D(\F)$.
\end{proposition}
\begin{proof}
  The algebras $\A_{\hbar_1,q}^D$ and $\A_{\hbar_2,q}^D$ consist of
  the same underlying vector space, however with two different
  products which we denote by $\cdot_{\hbar_1}$ and
  $\cdot_{\hbar_2}$. Next, let us define a map
  $\phi:\A_{\hbar_1,q}^D\to \A_{\hbar_2,q}^D$ by setting
  \begin{align*}
    \phi\parab{\sum_{k\in\ZD}f_kU^k} =
    \sum_{k\in\ZD}\paraa{T_{\hbar_1/\hbar_2}f_k}U^k,
  \end{align*}
  which is clearly linear and invertible (and well-defined since
  $T_{\hbar_1/\hbar_2}f_k\in \F$ by assumption), with
  \begin{align*}
    \phi^{-1}\parab{\sum_{k\in\ZD}f_kU^k} =
    \sum_{k\in\ZD}\paraa{T_{\hbar_2/\hbar_1}f_k}U^k.
  \end{align*}
  Now, let us prove that $\phi$ is an algebra homomorphism. First,
  compute (implicitly assuming summation over $k$ and $l$)
  \begin{align*}
    \phi\paraa{f_kU^k\cdot_{\hbar_1}g_lU^l}
    &= \phi\paraa{q^{N(k,l)}f_k(S_{\hbar_1}^kg_l)U^{k+l}}
      =q^{N(k,l)}(T_{\hbar_1/\hbar_2}f_k)
      \paraa{T_{\hbar_1/\hbar_2}S^k_{\hbar_1}g_l}U^{k+l}.
  \end{align*}
  Next, compute
  \begin{align*}
    \phi(f_kU^k)\cdot_{\hbar_2}\phi(g_lU^l)
    &= (T_{\hbar_1/\hbar_2}f_k)U^k\cdot_{\hbar_2}(T_{\hbar_1/\hbar_2}g_l)U^l\\
      &= q^{N(k,l)}(T_{\hbar_1/\hbar_2}f_k)
        \paraa{S_{\hbar_2}^kT_{\hbar_1/\hbar_2}g_l}U^{k+l}
        =\phi\paraa{f_kU^k\cdot_{\hbar_1}g_lU^l},
  \end{align*}
  by using that
  $S^k_{\hbar_2}T_{\hbar_1/\hbar_2}=T_{\hbar_1/\hbar_2}S^k_{\hbar_1}$. Furthermore,
  in a similar way, one can check that $\phi$ is a $\ast$-algebra
  homomorphism, showing that $\phi$ is indeed a $\ast$-algebra
  isomorphism.
\end{proof}

\noindent
Similarly, one can show that if $\F$ contains complex exponential functions,
then the shift algebras are isomorphic for all values of $q$.

\begin{proposition}\label{prop:q.isomorphism}
  Let $\A_{\hbar,q}(\F)$ be a shift algebra such that
  $e^{i\lambda\cdot u}f\in\F$ for all $f\in\F$ and
  $\lambda\in\reals^D$. Then
  $\A_{\hbar,q}^D(\F)\simeq\A^D_{\hbar,1}(\F)$.
\end{proposition}

\begin{proof}
  For $q=e^{i\hbar\tau}$ (with $\tau\in\reals$)
  define $\phi_q:\A^D_{\hbar,q}\to\A^D_{\hbar,1}$ as
  \begin{align*}
    \phi_q(f_kU^k) = f_ke^{i\tau N(u,k)}U^k
  \end{align*}
  Note that $\phi_q(f)\in\A^D_{\hbar,1}$ since, by assumption,
  $e^{i\lambda\cdot u}f_k\in\F$ for $\lambda\in\reals^D$. Furthermore,
  $\phi_q$ is clearly invertible. Let us now show that $\phi_q$ is an
  algebra homomorphism. To this end, we denote the products in
  $\A^D_{\hbar,q}$ and $\A^D_{\hbar,1}$ by $\cdot_q$ and $\cdot_1$,
  respectively.
  \begin{align*}
    \phi_q\paraa{f_kU^k\cdot_{q} g_lU^l}
    &= \phi_q\paraa{q^{N(k,l)}f_k(\Sh^k g_l)U^{k+l}}
      = q^{N(k,l)}e^{i\tau N(u,k+l)}f_k(\Sh^k g_l)U^{k+l}\\
    \phi_q(f_kU^k)\cdot_{1}\phi(g_lU^l)
    &= f_ke^{i\tau N(u,k)}(\Sh^kg_l)e^{i\tau N(u+k\hbar,l)}U^{k+l}\\
    &=e^{i\hbar\tau N(k,l)}e^{i\tau N(u,k+l)}f_k(\Sh^kg_l)U^{k+l}
    = \phi_q\paraa{f_kU^k\cdot_{q} g_lU^l}.
  \end{align*}
  Now, let us show that $\phi_q$ is also a $\ast$-homomorphism,
  denoting the involutions by $\ast_q$ and $\ast_1$, respectively:
  \begin{align*}
    \phi_q\paraa{(f_kU^k)^{\ast_q}}
    &=\phi_q\paraa{q^{N(k,k)}\Sh^k(\bar{f}_{-k})U^k}
      =q^{N(k,k)}e^{i\tau N(u,k)}\Sh^k(\bar{f}_{-k})U^k\\
    \phi_q(f_kU^k)^{\ast_1}
    &=\paraa{f_ke^{i\tau N(u,k)}U^k}^{\ast_1}
      = \Sh^k(\bar{f}_{-k})e^{i\tau N(u+k\hbar,k)}U^k\\
      &= e^{i\hbar\tau N(k,k)}e^{i\tau N(u,k)}\Sh^k(\bar{f}_{-k})U^k
        =\phi_q\paraa{(f_kU^k)^{\ast_q}},
  \end{align*}
  and we conclude that $\phi_q$ is a
  $\ast$-algebra isomorphism.
\end{proof}

\noindent
A particular example for which the above result does not apply is
when $\F$ is chosen to be the $\ast$-subalgebra of $\Fall$ consisting
of constant functions. For $D=2$ one then recovers the algebra of the
noncommutative torus, for which it is well-known that there is a family of
pairwise non-isomorphic algebras for different choices of $q$.

\section{Modules over shift algebras}\label{sec:modules.shift.algebras}

\noindent
Let us now introduce a class of left $\AhqDF$-modules. These modules
provide a generalization of the modules constructed in
\cite{al:projectors.cylinder} for the noncommutative cylinder, and
there are several novel results which may be applied to that case as
well.  As vector spaces, these modules consist of subspaces of
$F(\reals^D\times\integers^D,\complex)$ (the space of complex valued
functions on $\reals^D\times\integers^D$). However, one needs a few
assumptions on $\S$ which we present in the following definition.

\begin{definition}
  Let $\F$ be a $\ast$-subalgebra of $\Fall$. A subspace
  $\S\subseteq F(\reals^D\times\integers^D,\complex)$ is called
  \emph{$\F$-invariant} if $\xi(x,k)\in\S$ implies that
  $f(x)\xi(x,k)\in\S$, $\xi(x+\lambda,k+r)\in\S$, and
  $e^{i\lambda\cdot x}\xi(x,k)\in\S$ for all $\lambda\in\reals^D$,
  $r\in\integers^D$ and $f\in\F$.
\end{definition}

\noindent
Depending on the context (and the choice of $\F$), one can for
instance let $\S$ be the space of continuous functions, or simply let
$\S=F(\reals^D\times\integers^D,\complex)$. In the next result, a
class of left $\AhqDF$-modules is introduced, parametrized by a set of
real and integer matrices.
 
\begin{proposition}\label{prop:left.module}
  Let $\AhqDF$ be a shift algebra and let
  $\S\subseteq F(\reals^D\times\integers^D)$ be a $\F$-invariant
  subspace. If $\Lambda_0,\Lambda_1,E\in\Mat_{D}(\reals)$,
  $R\in\Mat_D(\integers)$ and $\delta\in\reals^D$ such that
  $\Lambda_0E+\Lambda_1R=\mid$, then $\S$ is a left
  $\AhqDF$-module with the module structure defined by
  \begin{align*}
    (f\xi)(x,k)=\sum_{n\in\integers^D}
    q^{N(\Phi(x,k,\delta),n)}
    f_n\paraa{\Phi(x,k,\delta)\hbar}    
    \xi(x+En,k+Rn)
  \end{align*}
  for $f=f_kU^k\in\AhqDF$, $\xi\in\S$ and
  $\Phi(x,k,\delta)=\Lambda_0x+\Lambda_1k+\delta$.
\end{proposition}

\begin{proof}
  The action is clearly linear, and it remains to show that
  \begin{align*}
    \paraa{(f_nU^n\cdot g_mU^m)\xi}(x,k)=
    \paraa{f_nU^n(g_mU^m\xi)}(x,k). 
  \end{align*}
  One computes
  \begin{align*}
    \paraa{(f_nU^n&\cdot g_mU^m)\xi}(x,k)
    =\paraa{(q^{N(n,m)}f_n(\Sh^ng_m)U^{n+m})\xi}(x,k)\\
                &=q^{N(n,m)}q^{N(\Phi(x,k,\delta),n+m)}
                  f_n(\Phi(x,k,\delta)\hbar)(\Sh^ng_m)(\Phi(x,k,\delta)\hbar)\times\\
                &\qquad\qquad\times\xi\paraa{x+E(n+m),k+R(n+m)}\\
                &=q^{N(n,m)}q^{N(\Phi(x,k,\delta),n+m)}
                  f_n(\Phi(x,k,\delta)\hbar)
                  g_{m}\paraa{\Phi(x,k,\delta)\hbar+n\hbar}\times\\
                &\qquad\qquad\times\xi\paraa{x+E(n+m),k+R(n+m)}.
  \end{align*}
  Next, one computes
  \begin{align*}
    \paraa{f_nU^n
    &(g_mU^m\xi)}(x,k)=q^{N(\Phi(x,k,\delta),n)}
      f_n(\Phi(x,k,\delta)\hbar)
      \paraa{(g_mU^m)\xi}(x+En,k+Rn)\\
    &=q^{N(\Phi(x,k,\delta),n)}
      f_n(\Phi(x,k,\delta)\hbar)
      q^{N(\Phi(x+En,k+Rn,\delta),m)}\times\\    
      &\qquad\times g_m\paraa{\Phi(x+En,k+Rn,\delta)\hbar}
        \xi\paraa{x+E(n+m),k+R(n+m)}\\
    &=q^{N(n,m)}q^{N(\Phi(x,k,\delta),n)}q^{N(\Phi(x,k,\delta),m)}f_n(\Phi(x,k,\delta)\hbar)\times\\
    &\qquad\times g_m\paraa{\Phi(x,k,\delta)\hbar+n\hbar}\xi\paraa{x+E(n+m),k+R(n+m)}\\
    &=\paraa{(f_nU^n\cdot g_mU^m)\xi}(x,k)
  \end{align*}
  by using that $\Phi(x+En,k+Rn,\delta)=\Phi(x,k,\delta)+n$, since
  $\Lambda_0E+\Lambda_1R=\mid$.
\end{proof}

\noindent
A left $\AhqDF$-module defined as in
Proposition~\ref{prop:left.module} will be denoted by
$\S_{\hbar}^\delta(\Lambda_0,E,\Lambda_1,R)$, or simply by $\S_\hbar^\delta$
(tacitly assuming a choice of $\Lambda_0,E,\Lambda_1,R)$.
Note that for $\Lambda_1=R=0$ and
$\Lambda_0E=\mid$, one effectively obtains a representation
acting on (a subalgebra of) $\Fall$ as
\begin{align}\label{eq:onlyx.rep}
  \paraa{f\xi}(x) = \sum_{n\in\integers^D}q^{N(\Lambda_0x+\delta,n)}
  f_n\paraa{(\Lambda_0x+\delta)\hbar}\xi(x+En).
\end{align}
Consequently, if $\F$ has the property that
$e^{i\lambda\cdot u}f\in\F$ for all $f\in\F$ and $\lambda\in\reals^D$,
then one may choose (with a slight abuse of notation) $\S=\F$ to
obtain a representation of $\AhqDF$. Note that, in this
case, Proposition~\ref{prop:q.isomorphism} implies that
$\AhqDF\simeq\A_{\hbar,1}^D(\F)$.  Similarly, for
$\Lambda_0=E=0$ and $\Lambda_1R=\mid$, one obtains a representation on
$F(\integers^D,\complex)$ given by
\begin{align}\label{eq:onlyn.rep}
  \paraa{f\xi}(k) = \sum_{n\in\integers^D}q^{N(\Lambda_1k+\delta,n)}
  f_n\paraa{(\Lambda_1k+\delta)\hbar}\xi(k+Rn).
\end{align}
In particular, for $\Lambda_1=R=\mid$ one obtains
\begin{align}\label{eq:onlyn.rep.R1}
  \paraa{f\xi}(k) = \sum_{n\in\integers^D}q^{N(k+\delta,n)}f_n\paraa{(k+\delta)\hbar}
  \xi(k+n).
\end{align}
A $\AhqDF$-module of the form
$\Shd(\Lambda_0,E,\Lambda_1,R)$ is defined by matrices
$\Lambda_0,\Lambda_1,E\in\Mat_D(\reals)$ and $R\in\Mat_D(\integers)$
together with $\delta\in\reals^D$. In the following, we try to
understand how these parameters affect the isomorphism class of the
module. The first result in this direction provides an isomorphism
between modules of the type $\Shd(\Lambda_0,E,\Lambda_1,R)$ for
different $\delta\in\reals^D$. Note that, in the following, when
comparing modules of the form $\Shd(\Lambda_0,E,\Lambda_1,R)$ for
different choices of parameters, we will (unless otherwise stated)
assume that the underlying $\F$-invariant subspace $\S$ is the same.

\begin{proposition}\label{prop:delta.isomorphism}
  If $\delta,\delta'\in\reals^D$ and $R\in\Mat_D(\integers)$, such
  that $R(\delta-\delta')\in\integers^D$, then
  \begin{align*}
   \S_\hbar^\delta(\Lambda_0,E,\Lambda_1,R)\simeq
    \S_\hbar^{\delta'}(\Lambda_0,E,\Lambda_1,R).
  \end{align*}
\end{proposition}

\begin{proof}
  Define
  $\phi:\S_{\hbar}^\delta(\Lambda_0,E,\Lambda_1,R)\to
  \S_{\hbar}^{\delta'}(\Lambda_0,E,\Lambda_1,R)$ as
  \begin{align*}
    \phi(\xi)(x,k)=\xi\paraa{x+E(\delta'-\delta),k+R(\delta'-\delta)}
  \end{align*}
  for $\xi\in\S$, which is well-defined since
  $R(\delta'-\delta)\in\integers^D$. Then $\phi$ is clearly an invertible
  linear map, and it remains to prove that it is a homomorphism. Thus,
  one computes 
  \begin{align*}
    \phi\paraa{(f_nU^n)\xi}(x,k)
    &=q^{N(\Phi(x+E(\delta'-\delta),k+R(\delta'-\delta),\delta),n)}
    f_n\paraa{\Phi(x+E(\delta'-\delta),k+R(\delta'-\delta),\delta)\hbar}\times\\
    &\qquad\times \xi\paraa{x+E(n+\delta'-\delta),k+R(n+\delta'-\delta)}\\
    &=q^{N(\Phi(x,k,\delta'),n)}f_n\paraa{\Phi(x,k,\delta')\hbar}\xi\paraa{x+E(n+\delta'-\delta),k+R(n+\delta'-\delta)}\\
    &=\paraa{(f_nU^n)\phi(\xi)}(x,k),
  \end{align*}
  by using that
  \begin{align*}
    \Phi(x+E(\delta'-\delta),k+R(\delta'-\delta),\delta)
    &=\Lambda_0x+\Lambda_1k+(\Lambda_0E+\Lambda_1R)(\delta'-\delta)+\delta\\
    &= \Lambda_0x+\Lambda_1k+\delta'-\delta+\delta=\Phi(x,k,\delta')
  \end{align*}
  since $\Lambda_0E+\Lambda_1R=\mid$.  We conclude that
  $\S_\hbar^\delta(\Lambda_0,E,\Lambda_1,R)\simeq\S_\hbar^{\delta'}(\Lambda_0,E,\Lambda_1,R)$.
\end{proof}

\noindent
In particular, if $\delta-\delta'\in\integers^D$ then
Proposition~\ref{prop:delta.isomorphism} implies that
\begin{align*}
  \Shd(\Lambda_0,E,\Lambda_1,R)\simeq
  \S_{\hbar}^{\delta'}(\Lambda_0,E,\Lambda_1,R) 
\end{align*}
for any choice of $\Lambda_0,E,\Lambda_1,R$ such that
$\Lambda_0E+\Lambda_1R=\mid$. Hence, up to module isomorphism, one
only needs to consider $\delta\in[0,1)^D$. Moreover,
Proposition~\ref{prop:delta.isomorphism} also implies that
\begin{align*}
  \Shd(\Lambda_0,E,\Lambda_1,0)\simeq
  \S_{\hbar}^{\delta'}(\Lambda_0,E,\Lambda_1,0)
\end{align*}
for arbitrary $\delta,\delta'\in\reals^D$. Having considered module
isomorphisms related to $\delta$, let us now focus on the matrices
defining $\Shd(\Lambda_0,E,\Lambda_1,R)$, and derive a sufficient
condition for isomorphic modules.

\begin{lemma}\label{lemma:isomorphism.conditions}
  Let $\S_{\hbar}^\delta(\Lambda_0,E,\Lambda_1,R)$ and
  $\S_{\hbar}^\delta(\Lt_0,\Et,\Lt_1,\Rt)$ be left
  $\A_{\hbar,q}^D(\F)$-modules. If there exist $A\in\GL_D(\reals)$
  and $B\in\GL_D(\integers)$ such that
  \begin{alignat}{2}
    &\Lambda_0A = \Lt_0 &\qquad  &A\Et = E\label{eq:iso.cond.LAE}\\
    &\Lambda_1B = \Lt_1 &\qquad  &B\Rt = R,\label{eq:iso.cond.LBR}
  \end{alignat}
  then $\Lambda_0E=\Lt_0\Et$, $\Lambda_1R=\Lt_1\Rt$ and 
  $\S_{\hbar}^\delta(\Lambda_0,E,\Lambda_1,R)\simeq\S_{\hbar}^\delta(\Lt_0,\Et,\Lt_1,\Rt)$.
\end{lemma}

\begin{proof}
  To distinguish the two different module structures on $\S$, let us
  denote the action of $f\in\AhqDF$ as $\rho(f)\xi$ and
  $\tilde{\rho}(f)\xi$ on $\Shd(\Lambda_0,E,\Lambda_1,R)$
  and $\Shd(\Lt_0,\Et,\Lt_1,\Rt)$, respectively.  Define a linear map
  $\phi:\S_{\hbar}^\delta(\Lambda_0,E,\Lambda_1,R)\to\S_{\hbar}^\delta(\Lt_0,\Et,\Lt_1,\Rt)$
  by
  \begin{align*}
    \phi(\xi)(x,k) = \xi(Ax,Bk)
  \end{align*}
  for $\xi\in\S$; since $A\in\GL_D(\reals)$ and
  $B\in\GL_D(\integers)$, $\phi$ is invertible. Let us show that
  \eqref{eq:iso.cond.LAE} and \eqref{eq:iso.cond.LBR} imply that
  $\phi$ is a left module homomorphism. First, writing
  \begin{align*}
    \Phi(x,k,\delta)=\Lambda_0x+\Lambda_1k+\delta\mathand
    \tilde{\Phi}(x,k,\delta)=\Lt_0x+\Lt_1k+\delta
  \end{align*}
  for $f=f_nU^n\in\A_{\hbar,q}^D(\F)$, one notes that
  \begin{align*}
    \Phi(Ax,Bk,\delta)=\Lambda_0Ax+\Lambda_1Bk+\delta
    &= \Lt_0x+\Lt_1k+\delta=\tilde{\Phi}(x,k,\delta)
  \end{align*}
  by using \eqref{eq:iso.cond.LAE} and \eqref{eq:iso.cond.LBR}. To
  show that $\phi$ is a homomorphism, one computes
  \begin{align*}
    \phi(\rho(f)\xi)(x,k)
    &= q^{N(\Phi(Ax,Bk,\delta),n)}
    f_n\paraa{\Phi(Ax,Bk,\delta)\hbar}\xi(Ax+En,Bk+Rn)\\
    &=q^{N(\tilde{\Phi}(x,k,\delta),n)}
    f_n(\tilde{\Phi}(x,k,\delta)\hbar)\xi(Ax+A\Et n,Bk+B\Rt n)
  \end{align*}
  and
  \begin{align*}
    \paraa{\tilde{\rho}(f)\phi(\xi)}(x,k)
    &=q^{N(\tilde{\Phi}(x,k,\delta),n)}
      f_n(\tilde{\Phi}(x,k,\delta)\hbar)\xi(Ax-A\Et n,Bk-B\Rt n)\\
    &=\phi(\rho(f)\xi)(x,k),
  \end{align*}
  by again using \eqref{eq:iso.cond.LAE} and \eqref{eq:iso.cond.LBR}. Hence,
  $\phi$ is a module isomorphism. Moreover, multiplying $A\Et=E$ from
  the left by $\Lambda_0$ and using that $\Lambda_0A=\Lt_0$ gives
  $\Lt_0\Et=\Lambda_0E$. Similarly, \eqref{eq:iso.cond.LBR} implies
  that $\Lambda_1R=\Lt_1\Rt$.
\end{proof}

\noindent
Let us now use the above result to show that under certain regularity
conditions of the parameters defining the module, a sufficient
condition for $\Shd(\Lambda_0,E,\Lambda_1,R)$ and
$\Shd(\Lt_0,\Et,\Lt_1,\Rt)$ to be isomorphic, is that
$\Lambda_0E=\Lt_0\Et$ (or, equivalently, $\Lambda_1R=\Lt_1\Rt$). More
precisely, we start by introducing the following compatibility
conditions.

\begin{definition}
  Let $\Lambda_0,\Lambda_1,E\in\Mat_{D}(\reals)$ and
  $R\in\Mat_{D}(\integers)$. The tuple $(\Lambda_0,E,\Lambda_1,R)$ is
  called \emph{regular} if $\Lambda_0$ or $E$ is invertible and
  $R\in\GL_D(\integers)$. Moreover, two tuples $(M_1,M_2,M_3,M_4)$ and
  $(\Mt_1,\Mt_2,\Mt_3,\Mt_4)$ are called \emph{compatible} if
  \begin{align*}
    M_i\text{ invertible }\equivalent \Mt_i\text{ invertible}
  \end{align*}
  for $i=1,2,3,4$.
\end{definition}

\begin{proposition}
  Let $\Shd(\Lambda_0,E,\Lambda_1,R)$ and $\Shd(\Lt_0,\Et,\Lt_1,\Rt)$
  be left $\A_{\hbar,q}^D(\F)$-modules such that
  $(\Lambda_0,E,\Lambda_1,R)$ and $(\Lt_0,\Et,\Lt_1,\Rt)$ are regular
  and compatible. If $\Lambda_0E=\Lt_0\Et$ then
  \begin{align*}
    \Shd(\Lambda_0,E,\Lambda_1,R)\simeq\Shd(\Lt_0,\Et,\Lt_1,\Rt). 
  \end{align*}
\end{proposition}

\begin{proof}
  Assume that $(\Lambda_0,E,\Lambda_1,R)$ and $(\Lt_0,\Et,\Lt_1,\Rt)$
  are regular and compatible. If $\Lambda_0,\Lt_0$ are invertible then
  one sets $A=\Lambda_0^{-1}\Lt_0$, and if $E,\Et$ are compatible then
  one sets $A=E\Et^{-1}$. Since $\Lambda_0E=\Lt_0\Et$ these choices of
  $A$ satisfy \eqref{eq:iso.cond.LAE} in
  Lemma~\ref{lemma:isomorphism.conditions}.  Note that, since
  $\Lambda_0E+\Lambda_1R=\Lt_0\Et+\Lt_1\Rt=\mid$ and
  $\Lambda_0E=\Lt_0\Et$, it follows that $\Lambda_1R=\Lt_1R$.  Since
  $R,\Rt$ are invertible, one can set $B=R\Rt^{-1}$, which is seen to
  satisfy \eqref{eq:iso.cond.LBR} by using that $\Lambda_1R=\Lt_1\Rt$.
  Lemma~\ref{lemma:isomorphism.conditions} then implies that
  $\Shd(\Lambda_0,E,\Lambda_1,R)\simeq\Shd(\Lt_0,\Et,\Lt_1,\Rt)$.
\end{proof}

\noindent
Recall that in the cases where $\Lambda_0=E=0$ or $\Lambda_1=R=0$ one
effectively obtains representations on functions on $\reals^D$ and
$\integers^D$, respectively (cf. equations \eqref{eq:onlyx.rep} and
\eqref{eq:onlyn.rep}). Let us now examine these two cases in
detail. The next results implies that for $\Lambda_1=R=0$, there is
only one equivalence class of isomorphic modules.

\begin{proposition}
  Let $\Shd(\Lambda_0,E,0,0)$ and
  $\S_{\hbar}^{\delta'}(\Lt_0,\Et,0,0)$ be left
  $\A_{\hbar,q}^D(\F)$-modules defined as in
  Proposition~\ref{prop:left.module}. Then
  $\Shd(\Lambda_0,E,0,0)\simeq\S_{\hbar}^{\delta'}(\Lt_0,\Et,0,0)$.
\end{proposition}

\begin{proof}
  It follows directly from Proposition~\ref{prop:delta.isomorphism}
  that
  $\S_{\hbar}^{\delta'}(\Lt_0,\Et,0,0)\simeq \Shd(\Lt_0,\Et,0,0)$,
  (since $\Rt=0$).  Next, let us show that
  $\Shd(\Lambda_0,E,0,0)\simeq\Shd(\Lt_0,\Et,0,0)$. Since
  $\Lambda_0E+\Lambda_1R=\Lambda_0E=\mid$ and
  $\Lt_0\Et+\Lt_1\Rt=\Lt_0\Et=\mid$, it follows that $\Lambda_0,\Lt_0$
  are invertible. Setting $A=\Lambda_0^{-1}\Lt_0$ and $B=\mid$ one can
  use Lemma~\ref{lemma:isomorphism.conditions}, to conclude that
  $\Shd(\Lambda_0,E,0,0)$ and $\Shd(\Lt_0,\Et,0,0)$ are
  isomorphic which, together with the previous argument, implies that
  $\Shd(\Lambda_0,E,0,0)\simeq\S_{\hbar}^{\delta'}(\Lt_0,\Et,0,0)$.
\end{proof}

\noindent
Thus, for these modules one may always choose $\Lambda_0=E=\mid$ and
$\delta=0$, giving
\begin{align}
  \paraa{f\xi}(x) = \sum_{n\in\integers^D}q^{N(x,n)}
  f_n(x\hbar)\xi(x+n\hbar).
\end{align}

\noindent
Let us now consider modules of the form $\Shd(0,0,\Lambda_1,R)$. We
note that specifying such a module amounts to choosing
$\delta\in\reals^D$ and a matrix $R\in\Mat_D(\integers)$ with
$\det R\neq 0$, giving the module $\Shd(0,0,R^{-1},R)$ (clearly
satisfying $\Lambda_1R=R^{-1}R=\mid$). Moreover, as previously discussed, in
the case when $\Lambda_0=E=0$ one may consider $\S$ to be a subset of
complex valued functions on $\integers^D$, and in the following, we
will consider the case when $\S$ consists of all \emph{compactly
  supported} such functions. To emphasize this particular setup, we
introduce the notation
\begin{align*}
  \Ssh(R,\delta) = \Shd(0,0,R^{-1},R).
\end{align*}
Furthermore, we let $\Matbar_D(\integers)\subseteq\Mat_D(\integers)$
denote the set of integer $(D\times D)$-matrices with nonzero
determinant. Note that a basis for (compactly supported) complex
valued functions on $\integers^D$ is given by
$\{\ket{k}\}_{k\in\integers^D}$ defined as
\begin{align*}
  \ket{k}(n) =
  \begin{cases}
    1\text{ if }k=n\\
    0\text{ if }k\neq n
  \end{cases}
\end{align*}
giving for $f=f_nU^n$
\begin{align}
  f\ket{k} = \sum_{n\in\integers^D}q^{N(R^{-1}k-n+\delta,n)}
  f_{n}\paraa{(R^{-1}k-n+\delta)\hbar}\ket{k-Rn}.
\end{align}
Let us begin by deriving sufficient conditions for
isomorphisms of the modules $\Ssh(R,\delta)$.

\begin{proposition}\label{prop:discrete.module.iso.cond}
  Let $R,\Rt\in\Matbar_D(\integers)$ and
  $\delta,\deltat\in\reals^D$. If
  \begin{align*}
    R\Rt^{-1}\in\GL_D(\integers)\qand
    R(\delta-\deltat)\in\integers^D 
  \end{align*}
  then $\Ssh(R,\delta)\simeq\Ssh(\Rt,\deltat)$.
\end{proposition}

\begin{proof}
  Under the assumption that $R\Rt^{-1}\in\GL_D(\integers)$, it follows
  directly from Lemma~\ref{lemma:isomorphism.conditions}, with
  $A=\mid$ and $B=R\Rt^{-1}$, that
  $\Ssh(R,\deltat)\simeq\Ssh(\Rt,\deltat)$. Furthermore, since
  $R(\delta-\deltat)\in\integers^D$,
  Proposition~\ref{prop:delta.isomorphism} implies that
  $\Ssh(R,\deltat)\simeq\Ssh(R,\delta)$.
\end{proof}

\noindent
Next, let us show that the conditions in
Proposition~\ref{prop:discrete.module.iso.cond} are also necessary if
the function algebra separates points. Recall that a subalgebra
$\F\subseteq\Fall$ \emph{separates points} if for any two distinct
points $x,y\in\reals^D$, there exists $f\in\F$ such that
$f(x)\neq f(y)$.

\begin{proposition}\label{prop:discrete.module.iso}
  Let $\AhqDF$ be a shift algebra such that $\F$ separates
  points. Then $\Ssh(R,\delta)\simeq\Ssh(\Rt,\delta)$ if and only if
  $R\Rt^{-1}\in\GL_D(\integers)$ and
  $R(\delta-\delta')\in\integers^D$.
\end{proposition}

\begin{proof}
  The sufficiency of the conditions given follows immediately from
  Proposition~\ref{prop:discrete.module.iso.cond}. Now, to show the
  necessity of the conditions, assume that
  \begin{align*}
    \phi:\Ssh(R,\delta)\to\Ssh(\Rt,\deltat)
  \end{align*}
  is a module isomorphism. With respect to the basis
  $\{\ket{k}\}_{k\in\integers^D}$, let us write
  \begin{align*}
    \phi(\ket{k})=\sum_{l\in\integers^D}\phi(k,l)\ket{l}
  \end{align*}
  with $\phi(k,l)\in\complex$. Note that, since $\phi$ is a vector
  space isomorphism, for each $k\in\integers^D$ there exist
  $k_1,k_2\in\integers^D$ such that $\phi(k_1,k)\neq 0$ and
  $\phi(k,k_2)\neq 0$. Furthermore, since $\phi$ is a module
  homomorphism one has $\phi(f\ket{k}) = f\phi(\ket{k})$ for all
  $f\in\F$, giving
  \begin{align*}
    &\sum_{l\in\integers^D}f\paraa{(R^{-1}k+\delta)\hbar}\phi(k,l)\ket{l}
    =\sum_{l\in\integers^D}f\paraa{(\Rt^{-1}l+\deltat)\hbar}\phi(k,l)\ket{l}\equivalent\\
    &\parab{f\paraa{(R^{-1}k+\delta)\hbar}
    -f\paraa{(\Rt^{-1}l+\deltat)\hbar}}\phi(k,l)=0
  \end{align*}
  for all $k,l\in\integers^D$ and $f\in\F$. Thus, for each
  $k\in\integers^D$ it follows that there exist $k_1,k_2$ (as introduced above) such that
  \begin{align*}
    &f\paraa{(R^{-1}k+\delta)\hbar}=f\paraa{(\Rt^{-1}k_2+\deltat)\hbar}\\
    &f\paraa{(R^{-1}k_1+\delta)\hbar}=f\paraa{(\Rt^{-1}k+\deltat)\hbar}
  \end{align*}
  for all $f\in\F$. Since the algebra $\F$ separates points, it
  follows that
  \begin{align*}
    &R^{-1}k+\delta = \Rt^{-1}k_2 + \deltat\\
    &R^{-1}k_1+\delta = \Rt^{-1}k + \deltat.
  \end{align*}
  Multiplying the equations above from the
  left with $\Rt$ and $R$, respectively, one obtains
  \begin{align}
    &\Rt R^{-1}k+\Rt(\delta-\deltat) = k_2\in\integers^D\label{eq:RtL1}\\
    &R\Rt^{-1}k+R(\deltat-\delta) = k_1\in\integers^D\label{eq:RLt1}
  \end{align}
  for all $k\in\integers^D$. First let us show that
  $R(\deltat-\delta)\in\integers^D$. If
  $R(\deltat-\delta)\notin\integers^D$ then it follows from
  \eqref{eq:RLt1} that $R\Rt^{-1}k\notin\integers^D$ for all
  $k\in\integers^D$ which contradicts the fact that
  $R\Rt^{-1}\in\Mat_D(\rationals)$. Hence,
  $R(\deltat-\delta)\in\integers^D$.  It then follows from
  \eqref{eq:RLt1} that $R\Rt^{-1}k\in\integers^D$ for all
  $k\in\integers^D$, which implies that
  $R\Rt^{-1}\in\Mat_D(\integers)$. A similar argument using
  \eqref{eq:RtL1} shows that $\Rt R^{-1}\in\Mat_D(\integers)$ from
  which we conclude that $R\Rt^{-1}\in\GL_D(\integers)$.
\end{proof}

\noindent The above result provides us with a necessary and sufficient
condition for modules of the form $\Ssh(R,\delta)$ to be isomorphic. In
the following, we shall obtain a more concrete description of the
equivalence classes of such modules, by finding representatives for
each equivalence class in a systematic way.  Specifying a module
$\Ssh(R,\delta)$ amounts to choosing $R\in\Matbar_D(\integers)$ and
$\delta\in\reals^D$ implying that the set of all such modules can be
parametrized by $\Matbar_D(\integers)\times\reals^D$. Let us now
introduce a group action $\alpha$ on
$\Matbar_D(\integers)\times \reals^D$ that will induce equivalence
classes corresponding to isomorphism classes of modules. Namely, for
the direct product of the groups $\GL_D(\integers)$ (with
multiplicative group structure) and $\integers^D$ (with additive group
structure), one sets for $U\in\GL_D(\integers)$ and $n\in\integers^D$
\begin{align*}
  &\alpha_{U,n}(R,\delta) = (UR,\delta+R^{-1}n)
\end{align*}
for $(R,\delta)\in\Matbar_D(\integers)\times\reals^D$. One may readily
check that this is indeed a (left) group action.  With $\Rt=UR$ and
$\deltat=\delta+R^{-1}n$, Proposition~\ref{prop:discrete.module.iso}
immediately implies that
\begin{align*}
  \Ssh(R,\delta)\simeq \Ssh\paraa{\alpha_{U,n}(R,\delta)}=\Ssh(UR,\delta+R^{-1}n),
\end{align*}
showing that the isomorphism class of the module $\Ssh(R,\delta)$ is
preserved by the group action. Hence, pairs $(R,\delta)$ in the same
orbit of the action correspond to isomorphic modules. The group action
$\alpha$ induces an equivalence relation $\sim_\alpha$ on
$\Matbar_D(\integers)\times\reals^D$, and we denote the set of orbits
by $\Matbar_D(\integers)\times\reals^D\slash \sim_\alpha$.

\begin{proposition}\label{prop:Shiso.U}
  Let $\Ahq(\F)$ be a shift algebra such that $\F$ separates points of
  $\reals^D$. Then $\Ssh(R,\delta)\simeq\Ssh(\Rt,\deltat)$ if and only
  if there exists $U\in\GL_D(\integers)$ and $n\in\integers^D$ such
  that $\Rt=UR$ and $\deltat=\delta+R^{-1}n$.
\end{proposition}

\begin{proof}
  As already noted, if $\Rt=UR$ and $\deltat=\delta+R^{-1}n$ then it
  follows from Proposition~\ref{prop:discrete.module.iso} that
  $\Ssh(R,\delta)\simeq\Ssh(\Rt,\deltat)$. Now, assume that
  $\Ssh(R,\delta)\simeq\Ssh(\Rt,\deltat)$.
  Proposition~\ref{prop:discrete.module.iso} implies that
  $R\Rt^{-1}\in\GL_D(\integers)$, which implies that there exists
  $U\in\GL_D(\integers)$ such that $R\Rt^{-1}=U^{-1}$, giving
  $\Rt=UR$. Furthermore, Proposition~\ref{prop:discrete.module.iso}
  implies that there exists $n\in\integers^D$ such that
  $R(\delta-\deltat)=-n$, giving $\deltat=\delta+R^{-1}n$.
\end{proof}

\noindent
Let $M(\Ssh)$ denote the set of isomorphism classes of modules of the
type $\Ssh(R,\delta)$, and let $[\Ssh(R,\delta)]$ denote the
equivalence class of $\Ssh(R,\delta)$ in
$M(\Ssh)$. Proposition~\ref{prop:Shiso.U} implies that the map
$\iota:\Matbar_D(\integers)\times\reals^D\slash \sim_\alpha\to
M(\Ssh)$ given by
\begin{align*} 
  &\iota\paraa{[(R,\delta)]} = [\Ssh(R,\delta)]
\end{align*}
is well-defined and, moreover, a bijection. In view of this
bijection, is it possible to find a natural representative
$(R,\delta)$ for each isomorphism class in $M(\Ssh)$?  Let us recall
the Hermite normal form of an integer matrix with nonzero determinant
(see e.g. \cite{m:decomposable.groups}).

\begin{definition}
  A matrix $H\in\Matbar_D(\integers)$ is said to be
  in \emph{Hermite normal form} if
  \begin{enumerate}
  \item $H$ is an upper triangular matrix,
  \item $H_{ii}>0$ for $1\leq i\leq D$,
  \item $0\leq H_{ji}<H_{ii}$ for $1\leq j<i\leq D$.
  \end{enumerate}
\end{definition}

\noindent
It is well-known that for any matrix $M\in\Matbar_D(\integers)$ there
exists a matrix $U\in\GL_D(\integers)$ and a \emph{unique} matrix
$H\in\Matbar_D(\integers)$ in Hermite normal form such that
$M=UH$. Combining this result with Proposition~\ref{prop:Shiso.U}, we
conclude that every equivalence class in
$\Matbar_D(\integers)\times\reals^D\slash\sim_\alpha$ has a 
representative $(R,\delta)$ with $R$ in Hermite normal form. More
precisely, we formulate this as follows.

\begin{proposition}\label{prop:Ssh.parametrization}
  Let
  \begin{align*}
    HR_D = \{(H,\delta)\in\Matbar_D(\integers)\times\reals^D:\delta\in[0,1)^D\text{ and }H\text{ is in Hermite normal form}\}.
  \end{align*}
  The map $\hat{\iota}:HR_D\to M(\Ssh)$ given by $\hat{\iota}(H,\delta)=[\Ssh(H,H^{-1}\delta)]$ is a bijection.
\end{proposition}

\begin{proof}
  Let us first show that $\hat{\iota}$ is surjective. Let
  $[\Ssh(R,\delta)]\in M(\Ssh)$, and let $H$ be the Hermite normal
  form of $R$ implying that there exists $U\in\GL_D(\integers)$ such
  that $R=UH$. It follows from Proposition~\ref{prop:Shiso.U} that
  $[\Ssh(R,\delta)]=[\Ssh(H,\delta)]$. Next, let $\delta_0$ be the
  fractional part of $H\delta$; i.e. the unique $\delta_0\in[0,1)^D$
  such that $H\delta=n+\delta_0$ for some $n\in\integers^D$. Since
  \begin{align*}
    \delta = H^{-1}\delta_0+H^{-1}n,
  \end{align*}
  Proposition~\ref{prop:Shiso.U} gives that
  $[\Ssh(H,\delta)]=[\Ssh(H,H^{-1}\delta_0)]$, implying
  $\hat{\iota}(H,\delta_0)=[\Ssh(R,\delta)]$. Next, let us show that
  $\hat{\iota}$ is injective. To this end we assume that
  $\hat{\iota}(H_1,\delta_1)=\hat{\iota}(H_2,\delta_2)$, which is
  equivalent to
  $\Ssh(H_1,H^{-1}\delta_1)\simeq\Ssh(H_2,H^{-1}\delta_2)$. From
  Proposition~\ref{prop:Shiso.U} it follows that there exists
  $U\in\GL_D(\integers)$ such that $H_2=UH_1$ and $n\in\integers^D$
  such that $H_2^{-1}\delta_2=H_1^{-1}\delta_1+H_1^{-1}n$. Firstly,
  since the Hermite normal form is unique, it follows that $H_1=H_2$
  giving that $H_1^{-1}\delta_2=H_1^{-1}\delta_1+H_1^{-1}n$ which is
  equivalent to $\delta_2=\delta_1+n$. Secondly, since
  $\delta_1,\delta_2\in[0,1)^D$, it follows that $n=0$ and
  $\delta_1=\delta_2$. We conclude that $H_1=H_2$ and 
  $\delta_1=\delta_2$, which implies that $\hat{\iota}$ is injective.
\end{proof}

\noindent
Thus, the above result tells us that to every equivalence class of
modules, one may assign a unique matrix in Hermite normal form,
together with a unique element of $[0,1)^D$. Conversely, it also gives
a concrete way to determine whether the modules $\Ssh(R,\delta)$ and
$\Ssh(\Rt,\deltat)$ are isomorphic or not, by comparing the Hermite
normal forms $H$ and $\tilde{H}$, of $R$ and $\Rt$ respectively, as
well as the fractional parts of $H\delta$ and $\tilde{H}\deltat$. As
an illustration, let us apply
Proposition~\ref{prop:Ssh.parametrization} to a particular
example. Namely, let us consider modules $\Ssh(R,\delta)$ such that
$R$ is diagonal; i.e.
\begin{align*}
  R = \diag(p_1,p_2,\ldots,p_D)
\end{align*}
with $0\neq p_i\in\integers$ for $i=1,\ldots,D$. Note that, since $R$
is diagonal, it is already in Hermite normal form. First, we note that
in order for $\Ssh(R,\delta)\simeq\Ssh(\Rt,\deltat)$, with
$\Rt=\diag(q_1,\ldots,q_D)$, it is necessary that $p_i=q_i$ for
$i=1,\ldots,D$ by the uniqueness of the Hermite normal form. For fixed
diagonal $R$, the isomorphism classes are parametrized by
$\delta=R^{-1}\delta_0$ for $\delta_0\in[0,1)^D$, giving
\begin{align*}
  \delta\in[0,1/p_1)\times[0,1/p_2)\times\cdots\times
  [0,1/p_D).
\end{align*}
Moreover, for arbitrary $R\in\GL_D(\integers)$ we note that
$\Ssh(R,\delta)\simeq \Ssh(\mid,\delta)$ since the Hermite normal form
of an invertible matrix is the identity matrix. Thus, the isomorphism
classes of such modules can be represented by $\Ssh(\mid,\delta)$ with
$\delta\in[0,1)^D$.

Now, having a more or less complete understanding of the isomorphism
classes at hand, let us study when the module
$\Ssh(R,\delta)$ is simple.

\begin{proposition}
  Let $\Ahq(\F)$ be a shift algebra such that $\F$ separates points of
  $\reals^D$. Then the module $\Ssh(R,\delta)$ is simple if and only if
  $R\in\GL_D(\integers)$.
\end{proposition}

\begin{proof}
  A module is simple if and only if every element of the module is
  cyclic.  Let us start by showing that if $R\notin\GL_D(\integers)$
  then there exists an element that is not cyclic. If
  $R\notin\GL_D(\integers)$ then there exists $l\in\integers^D$ such
  that $Rn\neq l$ for all $n\in\integers^D$ (i.e. $R$ is not
  surjective). Hence, the vector $\ket{0}$ is not cyclic since
  \begin{align*}
    \sum_{n\in\integers}f_nU^n\ket{0} =
    \sum_{n\in\integers}q^{N(\delta,n)}f_n(\delta\hbar)\ket{-Rn}\neq\ket{l}
  \end{align*}
  for any choice of $f_n\in\F$. The above argument shows that if
  $\Ssh(R,\delta)$ is simple, then $R\in\GL_D(\integers)$.

  Next, assume that $R\in\GL_D(\integers)$ and let
  \begin{align*}
    v=\sum_{k\in\integers^D}v_k\ket{k}
  \end{align*}
  be an arbitrary element of $\Ssh(R,\delta)$. Assuming $v$ to be
  non-zero, there exists $k_0\in\integers^D$ such that
  $v_{k_0}\neq 0$. Since $R\in\GL_D(\integers)$ one can form
  \begin{align}
    fU^{R^{-1}(k_0-l)}v = \sum_{k\in\integers^D}q^{N(R^{-1}k+\delta,R^{-1}(k_0-l))}
    f\paraa{(R^{-1}k+\delta)\hbar}v_k\ket{k-k_0+l)}.\label{eq:fURm1}
  \end{align}
  for arbitrary $f\in\F$ and $l\in\integers^D$.  Since $\F$ separates points, for
  every finite set $I\subseteq\reals^D$ and $\reals^D\ni x_0\notin I$
  there exists $f\in\F$ such that $f(x)=0$ for all $x\in I$ and
  $f(x_0)=1$. Hence, there exists a function $f_0\in\F$ such that
  \begin{align*}
    f_0\paraa{(R^{-1}k_0+\delta)\hbar} = 1\qand
    f_0\paraa{(R^{-1}k+\delta)\hbar} = 0
  \end{align*}
  for all $k\neq k_0$ such that $v_k\neq 0$ (which, by the compact
  support of $v$, is a finite set). From \eqref{eq:fURm1} it follows
  that
  \begin{align*}
    \frac{1}{v_{k_0}} q^{-N(R^{-1}k_0+\delta,R^{-1}(k_0-l))}
    f_0U^{R^{-1}(k_0-l)}v = \ket{l},
  \end{align*}
  implying that each $v\in\Ssh(R,\delta)$ is cyclic and, hence, that
  $\Ssh(R,\delta)$ is a simple module.
\end{proof}

\noindent
Thus, combined with the previous remark that
$\Ssh(R,\delta)\simeq\Ssh(\mid,\delta)$ for $R\in\GL_D(\integers)$,
all simple modules can be represented by $\Ssh(\mid,\delta)$ for
$\delta\in[0,1)$. Let us now show that for arbitrary
$R\in\Matbar_D(\integers)$, $\Ssh(R,\mid)$ is in general a direct sum
of such modules.

For $R\in\Matbar_D(\integers)$ we define an action $\alpha_R$ of
$\integers^D$ on $\integers^D$:
\begin{align*}
  \alpha_R(n)\cdot m = m+Rn,
\end{align*}
and it is easy to check that this is indeed a group action on
$\integers^D$. Hence, $\alpha_R$ splits $\integers^D$ into $N_R$
orbits. The integer $N_R$ is approximately equal to the number of
integer lattice points in the parallelotope spanned by the column
vectors of $R$, and bounded from above by the determinant of $R$. In
terms of the action $\alpha_R$, one can formulate the decomposition of
an arbitrary module in the following way.

\begin{proposition} \label{direct_sum_module}
  Let $R\in\Matbar_D(\integers)$ and $\delta\in\reals^D$. Then
  \begin{align*}
    \Ssh(R,\delta)\simeq
    \bigoplus_{k=1}^{N_R}\Ssh(\mid,\delta).
  \end{align*}
\end{proposition}

\begin{proof}
  Let $\hat{n}_1,\ldots\hat{n}_{N_R}\in\integers^D$ be elements of the
  disjoint orbits of the group action $\alpha_R$, respectively, and
  define for $\delta_k=R^{-1}\hat{n}_k+\delta$
  \begin{align*}
    &\phi:\bigoplus_{k=1}^{N_R}\Ssh(\mid,\delta_k)\to\Ssh(R,\delta)
  \end{align*}
  by
  \begin{align*}
    \phi:\bigoplus_{k=1}^{N_R}\sum_{m\in\integers^D}v_{k,m}\ket{m}
    \mapsto
    \sum_{k=1}^{N_R}\sum_{m\in\integers^D}v_{k,m}\ket{\hat{n}_k+Rm}
  \end{align*}
  where $v_{k,m}\in\complex$ for $1\leq k\leq N_R$ and
  $m\in\integers^D$. Let us now show that $\phi$ is a module
  homomorphism. One computes
  \begin{align*}
    \phi\parac{
    &f_nU^n\cdot \bigoplus_{k=1}^{N_R}\sum_{m\in\integers^D}v_{k,m}\ket{m}} =
    \phi\parac{
    \bigoplus_{k=1}^{N_R}\sum_{m\in\integers^D}
    q^{N(m+\delta_k,n)}f_n\paraa{(m+\delta_k)\hbar}v_{k,m}\ket{m-n}
    }\\
    &=\sum_{k=1}^{N_R}\sum_{m\in\integers^D}q^{N(m+\delta_k,n)}f_n\paraa{(m+\delta_k)\hbar}v_{k,m}
      \ket{\hat{n}_k+R(m-n)},
  \end{align*}
  and notes that
  \begin{align*}
    f_nU^n&\cdot\phi\parac{\bigoplus_{k=1}^{N_R}\sum_{m\in\integers^D}v_{k,m}\ket{m}}
    =f_nU^n\cdot\sum_{k=1}^{N_R}\sum_{m\in\integers^D}v_{k,m}\ket{\hat{n}_k+Rm}\\
          &=\sum_{k=1}^{N_R}\sum_{m\in\integers^D}q^{N(R^{-1}\hat{n}_k+m+\delta,n)}
            f_n\paraa{(R^{-1}\hat{n}_k+m+\delta)\hbar}v_{k,m}
            \ket{\hat{n}_k+R(m-n)}\\
    &=\phi\parac{
    f_nU^n\cdot \bigoplus_{k=1}^{N_R}\sum_{m\in\integers^D}v_{k,m}\ket{m}}
  \end{align*}
  since $\delta_k=R^{-1}\hat{n}_k+\delta$. Hence, $\phi$ is a left
  module homomorphism. Next, let us show that $\phi$ is
  surjective. Let $m\in\integers^D$ and let $k_0\in\{1,\ldots,N_R\}$ be such that
  $\hat{n}_{k_0}$ is in the same orbit as $m$, implying that there exists
  $m_0\in\integers^D$ such that $m=\hat{n}_{k_0}+Rm_0$. Then
  \begin{align*}
    \phi\paraa{\underbrace{0\oplus\cdots\oplus 0\oplus\ket{m_0}}_{k_0}\oplus 0\oplus\cdots\oplus 0}
    =\ket{\hat{n}_{k_0}+Rm_0} = \ket{m}
  \end{align*}
  implying that $\phi$ is surjective. Now, to show that $\phi$ is
  injective one assumes that
  \begin{align*}
    0=\phi\parac{\bigoplus_{k=1}^{N_R}\sum_{m\in\integers^D}v_{k,m}\ket{m}}
    =\sum_{k=1}^{N_R}\sum_{m\in\integers^D}v_{k,m}\ket{\hat{n}_k+Rm}.
  \end{align*}
  Since the orbits of $\alpha_R$ are disjoint, this implies that
  \begin{align*}
    \sum_{m\in\integers^D}v_{k,m}\ket{\hat{n}_k+Rm} = 0
  \end{align*}
  for $k=1,\ldots,N_R$. Moreover, since $\det R\neq 0$, it follows
  that $v_{k,m}=0$ for $k=1,\ldots,N_R$ and $m\in\integers^D$, giving
  \begin{align*}
    \bigoplus_{k=1}^{N_R}\sum_{m\in\integers^D}v_{k,m}\ket{m} = 0.
  \end{align*}
  Thus, we can conclude that $\phi$ is a module isomorphism implying that
  \begin{align*}
    \Ssh(R,\delta)\simeq
    \bigoplus_{k=1}^{N_R}\Ssh(\mid,\delta_k).    
  \end{align*}
  Finally, we note that $R(\delta_k-\delta)=\hat{n}_k\in\integers^D$
  which implies (by Proposition~\ref{prop:discrete.module.iso.cond}) that
  $\Ssh(\mid,\delta_k)\simeq\Ssh(\mid,\delta)$; hence, we have shown that
  $\Ssh(R,\delta)\simeq\Ssh(\mid,\delta)\oplus\cdots\oplus\Ssh(\mid,\delta)$.
\end{proof}

\noindent
Proposition~\ref{direct_sum_module} shows that the basic building
block of modules of the type $\Ssh(R,\delta)$ is given by
$\Ssh(\mid,\delta)$. Moreover, it is natural to ask whether or not
these modules are free, or projective? The next pair of results show
that, under the hypothesis that $\F$ separates points,
$\Ssh(R,\delta)$ is never a free module, and only projective if the
function algebra contains a $\delta$-like function.

\begin{proposition}\label{prop:SshRd.not.free}
  Let $\Ahq(\F)$ be a shift algebra such that $\F$ separates points of
  $\reals^D$. Then the module $\Ssh(R,\delta)$ is not free.
\end{proposition}
\begin{proof}
For an arbitrary set of elements $v_1,\dots,v_N\in\Ssh(R,\delta)$ write
\begin{align*}
  v_i=\sum_{k\integers^D} v_i^k\ket{k}   
\end{align*}
and set
\begin{align*}
  I_i=\{k\in\integers^D:v_i^k\neq 0\}
\end{align*}
Note that the $I_i$ are finite sets. Since $\F$ separates points, one
can find non-zero $f_i\in\F$ such that $f_i\paraa{(R^{-1}k+\delta)\hbar}=0$
for $k\in I_i$, implying that
\begin{align*}
  f_iv_i = \sum_{k\in\integers^D}f_i\paraa{(R^{-1}k+\delta)\hbar}v_i^k\ket{k}=0
\end{align*}
and, consequently
\begin{align*}
  \sum_{i=1}^N f_i v_i=0.
\end{align*}
Hence $v_1,\dots,v_N$ do not form a basis of $\Ssh(R,\delta)$. Since the
set of vectors was arbitrary, we conclude that $\Ssh(R,\delta)$ is not free.
\end{proof}

\begin{proposition}\label{prop:SshRd.projective}
  Let $\Ahq(\F)$ be a shift algebra such that $\F$ separates points of
  $\reals^D$. Then the module $\Ssh(R,\delta)$ is a finitely generated
  projective module if and only if there exists $p_0\in\F$ such that
  \begin{align*}
    p_0(u) =
    \begin{cases}
      1 &\text{if } u=\hbar\delta\\
      0 &\text{if } u\neq\hbar\delta
    \end{cases}
  \end{align*}
\end{proposition}

\begin{proof}
  We proceed by showing the statement for $R=\mid$. It then follows
  from Proposition~\ref{direct_sum_module} that the same statement
  holds for arbitrary $R$, since a direct sum of modules is projective
  if and only if each factor is projective.

  First, assume that $p_0\in\F$. Let us construct a map
  $\phi:\A p_0\to \Ssh(\mid,\delta)$ by setting
  \begin{align*}
    \phi(f_nU^np_0) = f_nU^n\ket{0}.
  \end{align*}
  To show that $\phi$ is well-defined, one needs to prove that
  if $f_nU^np_0=0$ then $f_nU^n\ket{0}=0$. Assuming $f_nU^np_0=0$ one
  finds that
  \begin{align*}
    f_n(u)U^np_0(u)=0\implies
    f_n(u)p_0(u+n\hbar)U^n = 0\implies
    f_n(u)p_0(u+n\hbar) = 0
  \end{align*}
  for all $n\in\integers^D$. By the definition of $p_0$ it follows
  that $f_n\paraa{(\delta-n)\hbar}=0$ for all $n\in\integers^D$. Thus, one obtains
  \begin{align*}
    f_nU^n\ket{0} = q^{N(\delta-n,n)}f_n\paraa{(\delta-n)\hbar}\ket{-n}=0
  \end{align*}
  showing that $\phi$ is indeed well-defined; moreover, $\phi$ is
  clearly a left module homomorphism. To show that $\phi$ is
  injective, one assumes that $\phi(f_nU^np_0)=0$, giving
  \begin{align*}
    q^{N(\delta-n,n)}f_n\paraa{(\delta-n)\hbar}\ket{-n} = 0\implies
    f_n\paraa{(\delta-n)\hbar}=0
  \end{align*}
  for all $n\in\integers^D$. It follows that
  \begin{align*}
    f_n(u)U^np_0(u) = f_n(u)p_0(u+n\hbar)U^n = 0 
  \end{align*}
  since $p_0(u+n\hbar)\neq 0$ only if $u=(\delta-n)\hbar$. Hence,
  $\phi$ is injective. To prove that $\phi$ is surjective one simply
  notes that
  \begin{align*}
    \phi\paraa{q^{N(n+\delta,n)}U^{-n}p_0} = q^{N(n+\delta,n)}q^{N(n+\delta,-n)}\ket{n}
    =\ket{n}
  \end{align*}
  implying that $\phi$ is surjective. Thus, if $p_0\in\F$ then
  $\Ssh(\mid,\delta)$ is isomorphic to $\A p_0$, showing that
  $\Ssh(\mid,\delta)$ is a finitely generated projective module.

  Next, let us assume that $\Ssh(\mid,\delta)$ is a finitely generated
  projective module, and show that $p_0\in\F$. Thus, we assume that
  there exists a left module isomorphism
  $\phi:\Ssh(\mid,\delta)\to\A^Np$ for some $N\geq 1$ and
  $p\in\Mat_N(\A)$. Since $\phi$ is an isomorphism, there exists a
  non-zero $m_0\in\A^Np$ such that $\phi(\ket{0})=m_0$. Let
  $\{e_i\}_{i=1}^N$ be a basis of $\A^N$, and write $m_0 = m_0^ie_i$
  with $m_0^i = (m_0^i)_nU^n$ for $i=1,\ldots,N$. Furthermore, since
  $\phi$ is a module homomorphism, one obtains
  \begin{align*}
    f(\hbar\delta)m_0 = \phi\paraa{f(\hbar\delta)\ket{0}}
    = \phi(f\ket{0})
    = f\phi(\ket{0})
    = fm_0
  \end{align*}
  for all $f\in\F$, implying that
  \begin{align*}
    \paraa{f(u)-f(\hbar\delta)}(m_0^i)_n(u) = 0
  \end{align*}
  for $f\in\F$, $i=1,\ldots,N$ and $n\in\integers^D$. Since
  $m_0\neq 0$ there exist $i_0\in\{1,\ldots,N\}$ and
  $n_0\in\integers^D$ such that $\tilde{m}=(m_0^{i_0})_{n_0}\neq
  0$. Since $\F$ separates points one can, for every
  $\hbar\delta\neq u_0\in\reals^D$, find $f_0\in\F$ such that
  $f_0(u_0)-f_0(\hbar\delta)\neq 0$, implying that $\tilde{m}(u)=0$
  for every $u\neq\hbar\delta$. Since $\tilde{m}\neq 0$, one
  necessarily has $\tilde{m}(\hbar\delta)\neq 0$. Thus, one can define
  $p_0(u)=\tilde{m}(u)/\tilde{m}(\hbar\delta)$, proving that $p_0\in\F$.
\end{proof}

\noindent
For instance, if $\F$ consists of continuous functions,
Proposition~\ref{prop:SshRd.projective} implies that $\Ssh(R,\delta)$
is not a finitely generated projective module. In fact, for such
function algebras, these modules are in some sense far from being
projective. This is made precise in the following result.

\begin{proposition}\label{prop:torsion.module}
  Let $\Ahq(\F)$ be a shift algebra. The module $\Ssh(R,\delta)$ is a
  torsion module if and only if there exists $f\in\F$ and
  $k_0\in\integers^D$ such that $fg\neq 0$ for all nonzero $g\in\F$, and
  $f((k_0+\delta)\hbar)=0$. 
\end{proposition}

\begin{proof}
  The module $\Ssh(R,\delta)$ is a torsion module if for every
  $v\in\Ssh(R,\delta)$ there exists a non zero divisor $f_v\in\Ahq(\F)$ such that
  $f_vv=0$. Now, let $f\in\F$ be a function fulfilling the assumptions
  of Proposition~\ref{prop:torsion.module}. Defining
  \begin{align*}
    f_k(u) = (\Sh^{k_0-k}f)(u) = f\paraa{u+(k_0-k)\hbar}
  \end{align*}
  (which is clearly in $\F$ since $\F$ is assumed to be
  $\hbar$-invariant) it follows that
  $f_k((k+\delta)\hbar)=0$. Consequently, for any finite set
  $I\subseteq\integers^D$ the function
  \begin{align*}
    f_I = \prod_{k\in I}f_k
  \end{align*}
  fulfills $f_I((k+\delta)\hbar)=0$ for all $k\in I$.  We note that
  $f_I$ is not a zero divisor as an element of $\Ahq(\F)$ since $f_Ig\neq 0$ for all
  $0\neq g\in\F$. Thus given arbitrary $v\in\Ssh(R,\delta)$ with
  \begin{align*}
    v = \sum_{k\in\integers^D}v_k\ket{k}
  \end{align*}
  we let $I=\{k\in\integers^D:v_k\neq 0\}$ (which is a finite set) and find that
  \begin{align*}
    f_I v = \sum_{k\in I}v_kf_I\paraa{(k+\delta)\hbar}\ket{k} = 0.
  \end{align*}
  since $f_I((k+\delta)\hbar)=0$ for all $k\in I$.  Since $v$ was
  chosen arbitrarily, we conclude that $\Ssh(R,\delta)$ is a torsion
  module.

  Conversely, assume that $\Ssh(R,\delta)$ is a torsion module. From
  Proposition~\ref{direct_sum_module} it then follows that
  $\Ssh(\mid,\delta)$ is also a torsion module. In particular, for
  each $k\in\integers^D$, there exists a non zero divisor
  $f_{k}=f_{kn}U^n\in\Ahq(\F)$ such that
  \begin{align*}
    f_k\ket{k} = \sum_{n\in\integers^D}q^{N(k-n+\delta,n)}f_{kn}\paraa{(k-n+\delta)\hbar}\ket{k-n} = 0,
  \end{align*}
  implying that
  \begin{align*}
    f_{kn}\paraa{(k-n+\delta)\hbar} = 0
  \end{align*}
  for all $n,k\in\integers^D$. Now, since $f_k$ is assumed to be a non
  zero divisor there exists $n_0\in\integers^D$ such that
  $f_{kn_0}g\neq 0$ for all $0\neq g\in\F$. Hence, $f_{kn_0}\in\F$ has
  the desired properties, which concludes the proof.
\end{proof}

\subsection{Bimodules}

\noindent
In the previous section we have studied the structure of a class of
left $\AhqDF$-modules. In the following, we shall give these modules
the structure of both a right module and a bimodule. Moreover, we will
show that when $\Lambda_0\in\GL_D(\complex)$ and
$R\in\GL_D(\integers)$, given certain conditions on the function space
$\S$, the module $\Shd(\Lambda_0,E,\Lambda_1,R)$ is a free module of rank 1.

Let us start by introducing a class of right $\AhqDF$-modules in close
analogy with the left modules in the previous section.

\begin{proposition}\label{prop:right.module}
  Let $\Ahq(\F)$ be a shift algebra and let
  $\S\subseteq F(\reals^D\times\integers^D)$ be a $\F$-invariant
  subspace. If $\Gamma_0,\Gamma_1,F\in\Mat_{D}(\reals)$,
  $P\in\Mat_D(\integers)$ and $\delta\in\reals^D$ such that
  $\Gamma_0E+\Gamma_1P=\mid$, then $\S$ is a right
  $\A_{\hbar,q}(\F)$-module with
  \begin{align}\label{eq:right.module.def}
    (\xi f)(x,k) = \sum_{n\in\integers^D}q^{N(\Psi(x,k,\delta-n),n)}
    f_n\paraa{\Psi(x,k,\delta-n)\hbar}
    \xi\paraa{x-Fn,k-Pn}
  \end{align}
  for $f=f_nU^n\in\A_{\hbar,q}^D$, $\xi\in\S$ and
  $\Psi(x,k,\delta)=\Gamma_0x+\Gamma_1k+\delta$.
\end{proposition}

\begin{proof}
  To show that \eqref{eq:right.module.def} defines a right module
  action, one needs to check that $\xi(fg)=(\xi f)g$ for
  $f,g\in\A_{\hbar,q}$ and $\xi\in\S$. One computes
  \begin{align*}
    \paraa{&\xi(fg)}(x,k)
    =\paraa{\xi(q^{N(n,m)}f_n(\Sh^ng_m)U^{m+n})}(x,k)\\
    &=q^{N(n,m)}q^{N(\Psi(x,k,\delta-n-m),n+m)}
      f_n(\Psi(x,k,\delta-n-m)\hbar)
      g_m(\Psi(x,k,\delta-m)\hbar)\times\\
    &\hspace{50mm}\times
      \xi\paraa{x-F(m+n),k-P(n+m)}\\
           &= q^{-N(m,n)-N(n,n)-N(m,m)}q^{N(\Psi(x,k,\delta),m+n)}
             f_n(\Psi(x,k,\delta-n-m)\hbar)\times\\
    &\hspace{30mm}\times
      g_m(\Psi(x,k,\delta-m)\hbar)\xi\paraa{x-F(m+n),k-P(n+m)},
  \end{align*}
  as well as
  \begin{align*}
    \paraa{&(\xi f)g}(x,k)
             = q^{N(\Psi(x,k,\delta-m),m)}
             g_m\paraa{\Psi(x,k,\delta-m)\hbar}
             (\xi f)(x-Fm,k-Pm)\\
           &= q^{N(\Psi(x,k,\delta-m),m)}
             q^{N(\Psi(x-Fm,k-Pm,\delta-n),n)}
             g_m\paraa{\Psi(x,k,\delta-m)\hbar}\times\\
           &\hspace{15mm}\times
             f_n\paraa{\Psi(x-Fm,k-Pm,\delta-n)\hbar}
             \xi\paraa{x-F(m+n),k-P(m+n)}
  \end{align*}
  Using that $\Gamma_0F+\Gamma_1P=\mid$, giving
  $\Psi(x-Fm,k-Pm,\delta)=\Psi(x,k,\delta-m)$, one obtains
  \begin{align*}
    \paraa{(\xi f)g}(x,k)
             &=q^{-N(m,n)-N(n,n)-N(m,m)}q^{N(\Psi(x,k,\delta),m+n)}
             g_m\paraa{\Psi(x,k,\delta-m)\hbar}\times\\
           &\hspace{10mm}\times
             f_n\paraa{\Psi(x,k,\delta-n-m)\hbar}
             \xi\paraa{x-F(m+n),k-P(m+n)}\\
    &=\paraa{\xi(fg)}(x,k),
  \end{align*}
  and we conclude that \eqref{eq:right.module.def} is a right module
  action on $\S$.
\end{proof}

\noindent
Thus, an $\F$-invariant subspace $\S$ carries both the structure of a
left module (as given in Proposition~\ref{prop:left.module}) and the
structure of a right module, as given above. The next result shows
that if the corresponding matrices defining the module structure are
compatible, then $\S$ is a bimodule.  More precisely, we formulate it
as follows.

\begin{proposition}\label{prop:bimodule}
  Let $\Ahq^D(\F)$ and $\A_{\hbar',q'}^D(\F')$ be shift algebras and
  let $\S\subseteq F(\reals^D\times\integers^D)$ be a subspace which
  is both $\F$-invariant and $\F'$-invariant. Furthermore,
  let $\Lambda_0,\Lambda_1,E,\Gamma_0,\Gamma_1,F\in\Mat_{D}(\reals)$,
  $R,P\in\Mat_D(\integers)$ and $\delta,\epsilon\in\reals^D$ such that
  \begin{alignat}{2}
    &\Lambda_0E+\Lambda_1R=\mid &\qquad  &\Gamma_0F+\Gamma_1P=\mid\label{eq:bimodule.cond.1}\\
    &\Lambda_0F+\Lambda_1P = 0 & &\Gamma_0E+\Gamma_1R =
    0.\label{eq:bimodule.cond.0}
  \end{alignat}
  Then $\S$ is a $\A_{\hbar,q}^D(\F)$-$\A_{\hbar',q'}^D(\F')$-bimodule with
  \begin{align}
    &(f\xi)(x,k)=\sum_{n\in\integers^D}
    q^{N(\Phi(x,k,\delta),n)}
    f_n\paraa{\Phi(x,k,\delta)\hbar}
    \xi(x+En,k+Rn)\\
    &(\xi g)(x,k) = \sum_{n\in\integers^D}(q')^{N(\Psi(x,k,\epsilon-n),n)}
    g_n\paraa{\Psi(x,k,\epsilon-n)\hbar'}
    \xi\paraa{x-Fn,k-Pn}
  \end{align}
  for $f=f_nU^n\in\A_{\hbar,q}^D$, $g=g_nU^n\in\A_{\hbar',q'}^D(\F')$,
  $\xi\in\S$, $\Phi(x,k,\delta)=\Lambda_0x+\Lambda_1k+\delta$ and
  $\Psi(x,k,\epsilon)=\Gamma_0x+\Gamma_1k+\epsilon$.
\end{proposition}

\begin{proof}
  We conclude from Proposition~\ref{prop:left.module} and
  Proposition~\ref{prop:right.module} that $\S$ is a left
  $\Ahq^D(\F)$-module and a right
  $\A_{\hbar',q'}^D(\F')$-module. Thus, it remains to show that
  $f(\xi g)=(f\xi)g$. One computes
  \begin{align*}
    \para{(f\xi)g}&(x,k) =
      (q')^{N(\Psi(x,k,\epsilon-m),m)}
      g_m\paraa{\Psi(x,k,\epsilon-m)\hbar'}(f\xi)(x-Fm,k-Pm)\\
    &=q^{N(\Phi(x-Fm,k-Pm,\delta),n)}(q')^{N(\Psi(x,k,\epsilon-m),m)}
      g_m\paraa{\Psi(x,k,\epsilon-m)\hbar'}\times\\
      &\qquad \times f_n\paraa{\Phi(x-Fm,k-Pm,\delta)\hbar}
        \xi(x-Fm+En,k-Pm+Rn)\\
    &=q^{N(\Phi(x,k,\delta),n)}(q')^{N(\Psi(x,k,\epsilon-m),m)}
      g_m\paraa{\Psi(x,k,\epsilon-m)\hbar'}\times\\
      &\qquad \times f_n\paraa{\Phi(x,k,\delta)\hbar}
        \xi(x-Fm+En,k-Pm+Rn)
  \end{align*}
  by using that $\Lambda_0F+\Lambda_1P=0$. On the other hand, one
  obtains
  \begin{align*}
    \paraa{f(\xi g)}
    &(x,k) =
      q^{N(\Phi(x,k,\delta),n)}
      f_n\paraa{\Phi(x,k,\delta)\hbar}
      (\xi g)(x+En,k+Rn)\\
    &=q^{N(\Phi(x,k,\delta),n)}
      (q')^{N(\Psi(x+En,k+Rn,\epsilon-m),m)}
      f_n\paraa{\Phi(x,k,\delta)\hbar}\times\\
    &\qquad \times g_m\paraa{\Psi(x+En,k+Rn,\epsilon-m)\hbar'}
      \xi(x-Fm+En,k-Pm+Rn)\\
    &=q^{N(\Phi(x,k,\delta),n)}
      (q')^{N(\Psi(x,k,\epsilon-m),m)}
      f_n\paraa{\Phi(x,k,\delta)\hbar}\times\\
    &\qquad \times g_m\paraa{\Psi(x,k,\epsilon-m)\hbar'}
      \xi(x-Fm+En,k-Pm+Rn)\\
    &=\para{(f\xi)g}(x,k)
  \end{align*}
  by using that $\Gamma_0E+\Gamma_1R = 0$. We conclude that $\S$ is a 
  $\A_{\hbar,q}^D(\F)$-$\A_{\hbar',q'}^D(\F')$-bimodule.
\end{proof}

\noindent
An immediate question arising from Proposition~\ref{prop:bimodule} is
whether or not one can find matrices satisfying conditions
\eqref{eq:bimodule.cond.1} and \eqref{eq:bimodule.cond.0}? The next
result gives an explicit solution to these equations in the case when
$\Lambda_0,\Gamma_0\in\GL_D(\complex)$ and $R,P\in\GL_D(\integers)$.

\begin{lemma}\label{lemma:bimodule.parameters}
  Let $\Lambda_0,\Gamma_0\in\GL_D(\complex)$,
  $\Lambda_1,\Gamma_1,E,F\in\Mat_D(\complex)$ and
  $R,P\in\GL_D(\integers)$. Then
  \begin{alignat*}{2}
    &\Lambda_0E+\Lambda_1R=\mid &\qquad  &\Gamma_0F+\Gamma_1P=\mid\\
    &\Lambda_0F+\Lambda_1P = 0 & &\Gamma_0E+\Gamma_1R = 0
  \end{alignat*}
  is equivalent to
  \begin{alignat*}{2}
    &\Gamma_0=-P^{-1}R\Lambda_0 &\qquad  &\Gamma_1=P^{-1}(\mid-R\Lambda_1)\\
    &E = \Lambda_0^{-1}(\mid-\Lambda_1R) & &F = -\Lambda_0^{-1}\Lambda_1P.
  \end{alignat*}
  Moreover, if $P=-R$ then the above system is equivalent to
  \begin{alignat*}{2}
    &\Gamma_0=\Lambda_0 &\qquad  &\Gamma_1=\Lambda_1-R^{-1}\\
    &E = \Lambda_0^{-1}(\mid-\Lambda_1R) & &F = \Lambda_0^{-1}\Lambda_1R.
  \end{alignat*}
\end{lemma}

\begin{proof}
  It follows immediately from
  \begin{align*}
    &\Lambda_0E+\Lambda_1R=\mid\qand
    \Lambda_0F+\Lambda_1P = 0
  \end{align*}
  that
  \begin{align}
    &E = \Lambda_0^{-1}(\mid-\Lambda_1R)\qand
    F = -\Lambda_0^{-1}\Lambda_1P\label{eq:lemma.E.F}
  \end{align}
  since $\Lambda_0\in\GL_D(\complex)$. Inserting these equations into
  \begin{align*}
    \Gamma_0F+\Gamma_1P=\mid\qand
    \Gamma_0E+\Gamma_1R = 0
  \end{align*}
  gives
  \begin{align*}
    \Gamma_1 = P^{-1}+\Gamma_0\Lambda_0^{-1}\Lambda_1\qand
    \Gamma_0\Lambda_0^{-1}(\mid-\Lambda_1R)+\Gamma_1R = 0
  \end{align*}
  which are equivalent to
  \begin{align}
    \Gamma_1 = P^{-1}(\mid-R\Lambda_1)\qand
    \Gamma_0 = -P^{-1}R\Lambda_0\label{eq:lemma.G0G1}
  \end{align}
  proving the first part of the statement. Finally, equations
  \eqref{eq:lemma.E.F} and \eqref{eq:lemma.G0G1} give
  \begin{alignat*}{2}
    &\Gamma_0=\Lambda_0 &\qquad  &\Gamma_1=\Lambda_1-R^{-1}\\
    &E = \Lambda_0^{-1}(\mid-\Lambda_1R) & &F = \Lambda_0^{-1}\Lambda_1R.
  \end{alignat*}
  when setting $P=-R$.
\end{proof}

\noindent
Thus, with the help of Lemma~\ref{lemma:bimodule.parameters} one
can easily construct matrices satisfying \eqref{eq:bimodule.cond.1}
and \eqref{eq:bimodule.cond.0}; e.g. for $\Lambda_0=\mid$, $R=-P=\mid$
and $\Lambda_1=2\cdot\mid$, one finds that
\begin{align*}
  &(\Lambda_0,E,\Lambda_1,R) = (\mid,-\mid,2\cdot\mid,\mid) \\
  &(\Gamma_0,F,\Gamma_1,P) = (\mid,2\cdot\mid,\mid,-\mid)
\end{align*}
define a bimodule structure on $\S$.  In
Section~\ref{sec:modules.shift.algebras} we studied the special cases
when $\Lambda_0=E=0$ or $\Lambda_1=R=0$ in detail. In the context of
bimodules, we note that there are no solutions of
\eqref{eq:bimodule.cond.1} and \eqref{eq:bimodule.cond.0} with
$\Lambda_0=\Gamma_0=0$ or $\Lambda_1=\Gamma_1=0$.  As previously
mentioned, when $\Lambda_0\in\GL_D(\complex)$ and
$R\in\GL_D(\integers)$, the module $\Shd(\Lambda_0,E,\Lambda_1,R)$
turns out to be a free module of rank 1. As a first step in proving
this statement, let us construct a module homomorphism
$\phi:\AhqDF\to\Shd(\Lambda_0,E,\Lambda_1,R)$.

\begin{proposition}\label{prop.algbra.homom.bimodule}
  Let $\Ahq^D(\F)$ be a shift algebra and let
  $\S\subseteq F(\reals^D\times\integers^D)$ be a $\F$-invariant
  subspace.  If $R\in\GL_D(\integers)$ and
  $f(\Phi(x,k,\delta)\hbar)\in\S$, for all $f\in\F$ and
  $\delta\in\reals^D$, then the map
  $\phi:\Ahq^D(\F)\to\Shd(\Lambda_0,E,\Lambda_1,R)$, defined as
  \begin{align}\label{eq:hom.alg.module}
    \phi(f_nU^n)(x,k) = q^{-N(\Phi(x,k,\delta),R^{-1}k)}
    f_{-R^{-1}k}\paraa{\Phi(x,k,\delta)\hbar},
  \end{align}
  is a left module homomorphism. Moreover, if $\S$ is a
  $\Ahq^D(\F)$-$\Ahq^D(\F)$-bimodule, as defined in
  Proposition~\ref{prop:bimodule}, with $P=-R$ and
  $\epsilon=\delta$, then $\phi$ is a bimodule homomorphism.
\end{proposition}

\begin{proof}
  Let us start by showing that $\phi$ is a left module
  homomorphism. To this end we compute
  \begin{align*}
    \phi&\paraa{f\cdot g}(x,k)
    = \sum_{n,l\in\integers^D}\phi\paraa{q^{N(l,n-l)}f_l(\Sh^lg_{n-l})U^n}(x,k)\\
    &= \sum_{l\in\integers^D}q^{-N(\Phi(x,k,\delta),R^{-1}k)}q^{N(l,-R^{-1}k-l)}f_l(\Phi(x,k,\delta)\hbar)
      g_{-R^{-1}k-l}(\Phi(x,k,\delta)\hbar+l\hbar),
  \end{align*}
  as well as
  \begin{align*}
    \paraa{f&\cdot\phi(g)}(x,k)
    = \sum_{l\in\integers^D}q^{N(\Phi(x,k,\delta),l)}
              f_l\paraa{\Phi(x,k,\delta)\hbar}\phi(g)(x+El,k+Rl)\\
            &= \sum_{l\in\integers^D}
              q^{N(\Phi(x,k,\delta),l)}
              f_l\paraa{\Phi(x,k,\delta)\hbar}
              q^{-N(\Phi(x+El,k+Rl,\delta),R^{-1}(k+Rl))}\times\\
            &\hspace{15mm}
              \times g_{-R^{-1}(k+Rl)}\paraa{\Phi(x+El,k+Rl,\delta)\hbar}\\
            &= \sum_{l\in\integers^D}q^{-N(\Phi(x,k,\delta),R^{-1}k)}q^{-N(l,R^{-1}k+l)}
              f_l\paraa{\Phi(x,k,\delta)\hbar}g_{-R^{-1}k-l}\paraa{\Phi(x,k,\delta)\hbar+l\hbar}\\
            &=\paraa{f\cdot\phi(g)}(x,k)
  \end{align*}
  by using that $\Phi(x+El,k+Rl,\delta)=\Phi(x,k,\delta)+l$. Hence,
  $\phi$ is a left module homomorphism.   Now,
  assume that $\S$ is a bimodule as defined in
  Proposition~\ref{prop:bimodule} with
  $R\in\GL_D(\integers)$, $P=-R$ and
  $\epsilon=\delta$. In this case, it follows from
  Lemma~\ref{lemma:bimodule.parameters} that
  \begin{align*}
    \Psi(x,k,\delta)=\Gamma_0x+\Gamma_1k+\delta
    =\Lambda_0x+\Lambda_1k-R^{-1}k+\delta=
    \Phi(x,k,\delta)-R^{-1}k.
  \end{align*}
  Let us now show that $\phi$ is also a right module homomorphism. One
  computes
  \begin{align*}
    \phi(g&\cdot f)(x,k)
            = \sum_{n,l\in\integers^D}\phi(q^{N(l,n-l)}g_l(\Sh^lf_{n-l})U^n)(x,k)\\
          &=\sum_{l\in\integers^D}q^{-N(\Phi(x,k,\delta),R^{-1}k)}q^{-N(l,R^{-1}k+l)}
            g_l\paraa{\Phi(x,k,\delta)\hbar}
            f_{-R^{-1}k-l}\paraa{\Phi(x,k,\delta)\hbar+l\hbar}
  \end{align*}
  as well as
  \begin{align*}
    \paraa{\phi(g)&\cdot f}(x,k)
                    = \sum_{n\in\integers^D}q^{N(\Psi(x,k,\delta-n),n)}
                    f_n(\Psi(x,k,\delta-n)\hbar)\phi(g)(x-Fn,k-Pn)\\
                  &=\sum_{n\in\integers}q^{N(\Psi(x,k,\delta-n),n)}
                    q^{-N(\Phi(x-Fn,k-Pn,\delta),R^{-1}(k-Pn))}
                    f_n(\Psi(x,k,\delta-n)\hbar)\times\\
                  &\hspace{20mm}\times
                    g_{-R^{-1}(k-Pn)}\paraa{\Phi(x-Fn,k-Pn,\delta)\hbar}
  \end{align*}
  which, by using that $\Lambda_0F+\Lambda_1P=0$ and $P=-R$, becomes
  \begin{align*}
    \sum_{n\in\integers}q^{N(\Psi(x,k,\delta-n),n)}
                    q^{-N(\Phi(x,k,\delta),R^{-1}k+n)}
    f_n(\Psi(x,k,\delta-n)\hbar)
    g_{-R^{-1}k-n}\paraa{\Phi(x,k,\delta)\hbar}
  \end{align*}
  By changing the summation index to $l=-R^{-1}k-n$, and using that
  $\Psi(x,k,\delta)+R^{-1}k=\Phi(x,k,\delta)$, one obtains
  \begin{align*}
    \sum_{l\in\integers^D}&
    q^{N(\Phi(x,k,\delta)+l,-R^{-1}k-l)}
    q^{-N(\Phi(x,k,\delta),-l)}
    f_{-R^{-1}k-l}(\Phi(x,k,\delta)\hbar+l\hbar)
    g_l\paraa{\Phi(x,k,\delta)\hbar}\\
                          &=\sum_{l\in\integers^D}
                            q^{N(\Phi(x,k,\delta),-R^{-1}k)}
                            q^{N(l,-R^{-1}k-l)}
                            f_{-R^{-1}k-l}(\Phi(x,k,\delta)\hbar+l\hbar)
                            g_l\paraa{\Phi(x,k,\delta)\hbar}\\
                          &=\phi(g\cdot f)(x,k),
  \end{align*}
  showing that $\phi$ is indeed a right module homomorphism under the
  above assumptions.
\end{proof}

\noindent
In principle, one can now show that if $\Lambda_0\in\GL_D(\complex)$,
then the homomorphism in Proposition~\ref{prop.algbra.homom.bimodule}
is in fact an isomorphism. However, there are a few technical
assumptions that one needs in order for the homomorphism to be
surjective. This statement is made precise in the next result.

\begin{proposition}\label{prop:bimodule.iso.free}
  Let $\S$ be a $\F$-invariant subspace such that
  \begin{enumerate}
  \item $\xi$ has compact support for all $\xi\in\S$,
  \item $\xi(Au+\lambda,n)\in\F$ for all $\xi\in\S$, $A\in\Mat_D(\complex)$,
    $\lambda\in\reals^D$ and $n\in\integers^D$,
  \item  $f(\Phi(x,k,\delta)\hbar)\in\S$, for all $f\in\F$ and
    $\delta\in\reals^D$.
  \end{enumerate}
  If $\Lambda_0\in\GL_D(\complex)$ then
  $\Shd(\Lambda_0,E,\Lambda_1,R)\simeq\AhqDF$ as a left
  module. Furthermore, if $\S$ is a
  $\Ahq^D(\F)$-$\Ahq^D(\F)$-bimodule, as defined in
  Proposition~\ref{prop:bimodule}, with $P=-R$ and $\epsilon=\delta$,
  then $\Shd(\Lambda_0,E,\Lambda_1,R)\simeq\AhqDF$ as a bimodule.
\end{proposition}

\begin{proof}
  Let us prove the statements by showing that the homomorphism in
  Proposition~\ref{prop.algbra.homom.bimodule} is a left module
  (resp. bimodule) isomorphism under the given assumptions. Since
  Proposition~\ref{prop.algbra.homom.bimodule} shows that $\phi$ is a
  module homomorphism, it remains to prove that $\phi$ is invertible.

  For $\Lambda_0\in\GL_D(\complex)$ one can explicitly construct the
  inverse of $\phi$ as
  \begin{align*}
    \phi^{-1}(\xi)(u)=
    \sum_{n\in\integers^D}q^{-\tfrac{1}{\hbar}N(u,n)}
    \xi\paraa{\Lambda_0^{-1}(u/\hbar+\Lambda_1Rn-\delta),-Rn}U^n,
  \end{align*}
  and we note that the sum is finite since $\xi$ has compact support,
  by assumption. Moreover, the above function is clearly in $\F$ since
  one assumes that $\xi(Au+\lambda,n)\in\F$ for all $\xi\in\S$,
  $A\in\Mat_D(\complex)$, $\lambda\in\reals^D$ and $n\in\integers^D$.
\end{proof}

\noindent
Note that the hypotheses of Proposition~\ref{prop:bimodule.iso.free}
are fulfilled if, for instance, $\F$ and $\S$ consists of all
functions with compact support on their respective domains.

\section{Representations of shift subalgebras}\label{sec:rep.subalgebras}

\noindent
From this section on we are interested in certain subalgebras of
$\AhqDF$, and their representations in order to construct fuzzy
analogues of a large class of level sets. We are particularly
interested in finite dimensional representations for which the vector
space dimension increases as the parameter $\hbar$ for the algebra
decreases. Such sequences of representations can be associated with a
level set, i.e. a subset of the form
\begin{align*}
 \{(x_1,\ldots,x_n)\in\reals^n : f(x_1,\ldots,x_n)=0\} 
\end{align*}
where $f:\reals^n\to\reals$ is a real valued function. We will also
consider sequences of infinite dimensional representations with a
countable basis, where the parameter $\hbar$ decreases.

For example, in the case $D=1$, one can associate algebra elements
$\sum f_n(u)U^n$ with functions $\sum_nf_n(u)e^{in\varphi}$. 
Then, we shall consider subalgebras generated by two
elements of the form $f(u)U$ and $f(u)U^{-1}$ and the classical limit
of such a subalgebra is a parametrized surface defined by
\begin{equation}
  \begin{split}
    & x=f(u)\cos \varphi \\
    & y=f(u)\sin \varphi \\
    & z=u.
  \end{split}\label{eq:1d.commutative.limit}
\end{equation}
Since the coordinates $x$ and $y$ fulfill
\begin{align} \label{immersion_f}
  & x^2+y^2= f(z)^2,
\end{align}
such equations describe an infinitely long cylinder immersed into
$\reals^3$, when $f(u)$ is real. We will see that for functions
$f(u)$, which become $0$ at some points and where the level set
pinches off, it is possible to construct representations, which
restrict to an interval of $\reals$, where $f(u)\geq0$ for
$u_1 \leq u \leq u_2$ (where possibly $u_1=-\infty$ and/or
$u_2=+\infty$). The classical limit of the construction is an
immersion of a finite cylinder or an infinite half-cylinder in
$\reals^3$. Let us now start by defining the subalgebras we are
interested in. As a standing assumption, we assume that $\F$ separates
points of $\reals^D$.

\begin{definition}\label{def:shift.subalgebra}
  A $\ast$-subalgebra $\A\subseteq\AhqD(\F)$ such that $fa\in\A$ for every
  $f\in\F$ and $a\in\A$ is called \emph{shift subalgebra}.
\end{definition} 

\noindent
In what follows, recall that $\{\ket{n}\}_{n\in\integers^D}$ is a basis of
$\Ssh(\mid,\delta)$, which is also a (left) module for every subalgebra
of $\AhqDF$, and that the action of an algebra element is
given as
\begin{align} \label{module.action}
  f_nU^n\ket{k} = \sum_{n\in\integers^D}q^{N(k-n+\delta,n)}
  f_{n}\paraa{(k-n+\delta)\hbar}\ket{k-n}.
\end{align}
For a function $f\in\F$ and the generators $U^n$ this reduces to
\begin{align*}
   f\ket{k} = f\paraa{(k+\delta)\hbar}\ket{k}\qand
   U^n\ket{k} = q^{N(k-n+\delta,n)}\ket{k-n}.
\end{align*}
In the case when $\F$ separates points, one can show that cyclic
modules, generated by a shift subalgebra acting on $\ket{k}$,
are simple.

\begin{proposition} \label{prop:AhqDFan.simple.module}
  Assume that $\F$ separates points and let $\A$ be a
  shift subalgebra of $\AhqD(\F)$. For $k\in\integers^D$, let
  $\Shd(\ket{k})\subseteq\S_{\hbar}(\mid,\delta)$ denote the
  $\A$-module generated by $\ket{k}$. Then $\Shd(\ket{k})$ is a simple
  $\A$-module.
\end{proposition}

\begin{proof}
  Let us show that $\Shd(\ket{k})$ is a simple module by showing that
  every vector is cyclic. We achieve this by showing that for every
  $v\in\Shd(\ket{k})$, there exists $a\in\A$ such that $av=\ket{k}$.

  To this end, let $v\in\Shd(\ket{k})$ be an arbitrary (nonzero) element. Since
  the module is generated by $\ket{k}$, there exists $a_v\in\A$ such
  that $v=a_v\ket{k}$. Writing
  \begin{align*}
    a_v = \sum_{n\in I}a_{v,n}U^n
  \end{align*}
  with $I\subseteq\integers^D$ such that $|I|<\infty$, one obtains
  \begin{align*}
     v=\sum_{n\in I} q^{N(k-n+\delta,n)} a_{v,n} \paraa{(k-n+\delta)\hbar} \ket{k-n},
  \end{align*}
  and, since $v\neq 0$, there exists $n_0\in I$ such that
  $a_{v,n_0} \paraa{(k-n_0+\delta)\hbar}\neq 0$. Now, since $\F$ is
  assumed to separate points, there exist $p_1,p_2\in\F$ such that
  \begin{align*}
     &p_1\paraa{(k-m+\delta)\hbar} =
     \begin{cases}
        1 &\text{if } m \in I\text{ and }m=n_0 \\
        0 &\text{if } m \in I\text{ and } m\neq n_0
      \end{cases}\\
    &p_2\paraa{(k-n_0+m+\delta)\hbar} =
     \begin{cases}
        1 &\text{if } m \in I\text{ and }m=n_0 \\
        0 &\text{if } m \in I\text{ and } m\neq n_0
      \end{cases},
  \end{align*}
  giving
  \begin{align*}
    p_1v = q^{N(k-n_0+\delta,n_0)}a_{v,n_0}\paraa{(k-n_0+\delta)\hbar}\ket{k-n_0}.
  \end{align*}
  Since $\A$ is a $\ast$-subalgebra, it follows that $a_v^\ast\in\A$ and
  \begin{align*}
    &a_v^\ast p_1v = q^{N(k-n_0+\delta,n_0)}a_{v,n_0}\paraa{(k-n_0+\delta)\hbar}\times\\
    &\quad\qquad\qquad\times\sum_{n\in I}q^{-N(k-n_0+n+\delta,n)}\overline{a_{v,n}\paraa{(k-n_0+\delta)\hbar}}\ket{k-n_0+n}\\
    &p_2a_v^\ast p_1v = q^{-N(n_0,n_0)}|a_{v,n_0}\paraa{(k-n_0+\delta)\hbar}|^2\ket{k}.
  \end{align*}
  Thus, setting
  \begin{align*}
    a = q^{N(n_0,n_0)}|a_{v,n_0}\paraa{(k-n_0+\delta)\hbar}|^{-2}p_2a_v^\ast p_1,
  \end{align*}
  it follows that $av = \ket{k}$. Hence, we have
  shown that every $v\in\Shd(\ket{k})$ is cyclic and, consequently, that
  $\Shd(\ket{k})$ is a simple module.
  
\end{proof}

\noindent
Recall that if the function algebra $\F$ separates points, then
Proposition~\ref{prop:SshRd.not.free} implies that the module $\Shd(\ket{k})$ is not free. 
Furthermore, Proposition~\ref{prop:SshRd.projective}  implies that in the case
when $\F$ separates points,  the module $\Shd(\ket{k})$ is projective, if and only
if there exists $p_0\in\F$ such that
\begin{align*}
    p_0(u) =
 \begin{cases}
      1 &\text{if } u=\hbar\delta\\
      0 &\text{if } u\neq\hbar\delta
    \end{cases}
\end{align*}
In the case of shift subalgebras generated from a finite set 
of algebra elements, we shall see that for particular choices,
modules defined in Proposition~\ref{prop:AhqDFan.simple.module} might
be finite dimensional. Moreover, by carefully choosing $\hbar$ in
relation to the dimension of the representation, one can construct
fuzzy spaces, i.e. sequences of finite dimensional representations of
increasing dimension as $\hbar$ tends to zero.

\subsection{Shift subalgebras for $D=1$} \label{sec:shift.subalgebras.one.dimension}

\noindent
For $f\in\F$, we let $\A_{\hbar}(\F,f)$ denote the shift subalgebra
generated by
\begin{align*}
  A_+=f \left( u-\tfrac{1}{2}\hbar \right) U^{-1}\qand
  A_-= f \left( u+\tfrac{1}{2}\hbar \right) U.
\end{align*}
The generators $A_+$ and $A_-$ satisfy
\begin{align}
  & [A_+,A_-] = f(u-\tfrac{1}{2}\hbar)^2
    -f(u+\tfrac{1}{2}\hbar)^2 \label{A_com} \\
  & A_+A_- + A_-A_+ = f(u-\tfrac{1}{2}\hbar)^2
    +f(u+\tfrac{1}{2}\hbar)^2, \label{A_anticom} 
\end{align}
and the action of $A_+$ and $A_-$ is given by
\begin{align} \label{A+A-action}
  A_+ \ket{n}=f \left( (n+\tfrac{1}{2}+\delta)\hbar \right)\ket{n+1} \qand
  A_- \ket{n}= f \left( (n-\tfrac{1}{2}+\delta)\hbar\right) \ket{n-1}.
\end{align}

\begin{proposition}\label{prop:1d-cases}
  For $f\in\F$ and $k\in\integers$, let $\Shd(\ket{k})$ denote the
  simple left $\A_{\hbar}(\F,f)$-module defined in
  Proposition~\ref{prop:AhqDFan.simple.module}.
  \begin{enumerate}
  \item If $f\paraa{(N+\tfrac{1}{2}+\delta)\hbar}=0$ for some
    $N\in\integers$ and $f\paraa{(n+\tfrac{1}{2}+\delta)\hbar}\neq 0$ for
    all integers $n< N$ then
    \begin{align*}
      \Shd(\ket{N}) =\Big\{\sum_{n\leq N}v_n\ket{n}:
      v_n\in\complex\text{ and }|n:v_n\neq 0|<\infty \Big\}.
    \end{align*}
  \item If $f\paraa{(M-\tfrac{1}{2}+\delta)\hbar}=0$, for some
    $M\in\integers$, and $f\paraa{(n-\tfrac{1}{2}+\delta)\hbar}\neq 0$ for
    all integers $n>M$ then
    \begin{align*}
      \Shd(\ket{M}) =\Big\{\sum_{n\geq M}v_n\ket{n}:
      v_n\in\complex\text{ and }|n:v_n\neq 0|<\infty \Big\}.
    \end{align*}
  \item If $f\paraa{(M-\tfrac{1}{2}+\delta)\hbar}=0$ 
    and $f\paraa{(N+\tfrac{1}{2}+\delta)\hbar}=0$, for some 
    $M,N\in\integers$ with $M<N$, and $f\paraa{(n+\tfrac{1}{2}+\delta)\hbar}\neq 0$ for
    all integers $M\leq n \leq N$ then
    \begin{align*}
      \Shd(\ket{M}) = \Shd(\ket{N}) =\Big\{\sum_{M\leq n\leq N}v_n\ket{n}:
      v_n\in\complex\Big\}.
    \end{align*}\label{item:1d.cases.finite}
  \end{enumerate}
\end{proposition}

\begin{proof}
  Let us provide a proof of (\ref{item:1d.cases.finite}), since the
  other two cases are proved analogously. The assumptions immediately
  gives that
  \begin{align*}
    &A_+\ket{N} = f\paraa{(N+\tfrac{1}{2}+\delta)\hbar}\ket{N+1} = 0\\
    &A_-\ket{M} = f\paraa{(M-\tfrac{1}{2}+\delta)\hbar}\ket{M-1} = 0.
  \end{align*}
  Moreover,
  \begin{align*}
    A_+\ket{n} &= f\paraa{(n+\tfrac{1}{2}+\delta)\hbar}\ket{n+1}\neq 0
      \quad\text{ for }M\leq n<N\\
    A_-\ket{n} &= f\paraa{(n-\tfrac{1}{2}+\delta)\hbar}\ket{n-1}\\
               &= f\paraa{(n-1+\tfrac{1}{2}+\delta)\hbar}\ket{n-1}\neq 0
      \quad\text{ for }M< n\leq N
  \end{align*}
  since $f\paraa{(n+\tfrac{1}{2}+\delta)\hbar}\neq 0$ for $M\leq n \leq N$,
  by assumption. It follows that
  \begin{equation*}
    \Shd(\ket{M}) = \S_{\hbar}(\ket{N}) =\Big\{\sum_{M\leq n\leq N}v_n\ket{n}:
    v_n\in\complex\Big\}.\qedhere
  \end{equation*}
\end{proof}

\noindent
Note that Proposition~\ref{prop:AhqDFan.simple.module} implies that
the modules in the above result are simple. In the following we shall
mostly be interested in the type of finite dimensional representations
described by (\ref{item:1d.cases.finite}) in Proposition~\ref{prop:1d-cases}.

\begin{example}\label{ex:3d_rep} 
\noindent  Let $\F$ be the polynomial functions on $\reals$ and let
  $f(u)=\frac{9}{4}-u^2$. Let us now construct a three dimensional
  representation of $\A_{\hbar}(\F,f)$ with $\hbar=1$. Since 
  \begin{align*}
    f \parab{\hbar(1+\frac{1}{2})}=f \parab{\hbar(-1-\frac{1}{2})}=0
  \end{align*}
  we choose $\delta=0$, $M=-1$ and $N=1$ and $\S_{1}^0(\ket{1})$ is
  spanned by $\ket{-1},\ket{0}, \ket{1}$ according to
  Proposition~\ref{prop:1d-cases}.  One finds that
  \begin{align*}
    &A_+\ket{-1} = f(-1+\tfrac{1}{2})\ket{0} = 2\ket{0}\quad
      A_+\ket{0} = f(0+\tfrac{1}{2})\ket{1} = 2\ket{1}\\
    &A_-\ket{1} = f(1-\tfrac{1}{2})\ket{0} = 2\ket{0}\quad
      A_-\ket{0} = f(0-\tfrac{1}{2})\ket{-1} = 2\ket{-1}
  \end{align*}
  Identifying
  \begin{align*}
    \ket{-1} \sim
    \begin{pmatrix}
      1 \\ 0 \\ 0
    \end{pmatrix}
    , \quad \ket{0} \sim
    \begin{pmatrix}
      0 \\ 1 \\ 0
    \end{pmatrix}
    , \quad \ket{1} \sim
    \begin{pmatrix}
      0 \\ 0 \\ 1
    \end{pmatrix}
  \end{align*}
  gives the matrix representation
  \begin{align*}
    A_+ =	\begin{pmatrix}
      0 & 0  & 0 \\
      2 & 0 & 0 \\
      0 & 2 & 0 
    \end{pmatrix} \qquad  
              A_- =	\begin{pmatrix}
                0 & 2  & 0 \\
                0 & 0 & 2 \\
                0 & 0 & 0 
              \end{pmatrix}.
  \end{align*}
  Moreover, $g\in\F$ is represented as the diagonal matrix $\hat{g}$,
  given by
  \begin{align*}
    \hat{g}=
    \begin{pmatrix}
      g(-1) & 0  & 0 \\
      0 & g(0) & 0 \\
      0 & 0 & g(1) 
    \end{pmatrix}.               
  \end{align*}
\end{example}

\noindent
Let us now describe how one can construct a sequence of algebras
\begin{align*}
  \A_{\hbar_1}(\F,f), \A_{\hbar_2}(\F,f),\A_{\hbar_3}(\F,f)\ldots 
\end{align*}
together with simple finite dimensional modules of increasing
dimension, providing a ``fuzzy'' analogue of the surface defined by
\eqref{eq:1d.commutative.limit}. To this end, assume that $f\in\F$
such that $f(u_1)=f(u_2)=0$, where $u_1<u_2$, and $f(u)\neq 0$ for
$u_1<u<u_2$. Furthermore, let $N_1,N_2,\ldots$ be an increasing
sequence of positive integers and set
\begin{align} \label{hbar.inverse.N}
  \hbar_k = \frac{u_2-u_1}{N_k}.
\end{align}
Choosing $M_k\in\integers$ and $\delta_k\in[0,1)$ such that
\begin{align*}
  u_1 = \paraa{M_k-\tfrac{1}{2}+\delta_k}\hbar
\end{align*}
one checks that
\begin{align*}
  &f\paraa{(M_k-\tfrac{1}{2}+\delta_k)\hbar} = f(u_1) = 0\\
  &f\paraa{(M_k+N_k-1+\tfrac{1}{2}+\delta_k)\hbar} = f(u_2) = 0
\end{align*}
and Proposition~\ref{prop:1d-cases} implies that
$\S_{\hbar_k}^{\delta_k}(\ket{M_k})$ is a simple $N_k$-dimensional
$\A_{\hbar_k}(\F,f)$-module. Clearly, because of (\ref{hbar.inverse.N}),
$\hbar_k$ decreases as $N_k$ increases. 

\begin{example}\label{ex:1d.quadratic.odd}
  As an illustration of the procedure given above, let us construct a
  sequence of odd-dimensional modules related to the algebra given in
  Example \ref{ex:3d_rep}.  In the notation above, let $u_1=-3/2$,
  $u_2=3/2$ giving $u_2-u_1=3$. For a sequence of positive integers
  $N_1<N_2<\cdots$ we set $\hbar_k=3/N_k$, and find $M_k\in\integers$
  and $\delta_k\in[0,1)$ such that
  \begin{align*}
    u_1 = \paraa{M-\tfrac{1}{2}+\delta_k}\hbar\equivalent
    -\frac{3}{2} = \paraa{M_k-\tfrac{1}{2}+\delta_k}\frac{3}{N_k}
  \end{align*}
  which results in
  \begin{align*}
    \begin{cases}
      M_k=\frac{1-N_k}{2}, \, \delta_k=0, 	& \text{for } N_k \text{ odd} \\
      M_k=-\frac{N_k}{2}, \, \delta_k=\frac{1}{2}, 	& \text{for } N_k \text{ even}.
    \end{cases}
  \end{align*} 
  With these values, $\S_{\hbar_k}^{\delta_k}(\ket{M_k})$ is a simple $N_k$-dimensional
  $\A_{3/N_k}(\F,f)$-module for $k\geq 1$. The matrix elements of the
  operators $A_+,A_-$ are easily deduced from
  \begin{align*}
    &A_+\ket{n} = f\paraa{(n+\tfrac{1}{2}+\delta_k)\hbar_k}\ket{n+1}
      = \paraa{\tfrac{9}{4}-(n+\tfrac{1}{2}+\delta_k)^2\tfrac{9}{N_k^2})}\ket{n+1}\\
    &A_-\ket{n} = f\paraa{(n-\tfrac{1}{2}+\delta_k)\hbar_k}\ket{n-1}
      = \paraa{\tfrac{9}{4}-(n-\tfrac{1}{2}+\delta_k)^2\tfrac{9}{N_k^2})}\ket{n-1}.
  \end{align*}
  It is interesting that we see a ``spin'' in $\delta_k$,
  related to even and odd dimension of the module.
  
  In the classical limit, the algebra elements $A_+$ and $A_-$ 
  are identified with the complexified coordinate
  $x+iy$ and its conjugate. The coordinates $x$ and $y$ then fulfill
  \begin{align*} 
    & x^2+y^2= \left( \frac{9}{4}-u^2 \right)^2 
  \end{align*}
  where $u$ is restricted to the interval $[u_1=-\frac{3}{2},u_2=\frac{3}{2}]$. The resulting
  surface is spindle shaped with peaks at $u_1$ and $u_2$ (cf. Figure~\ref{fig:spindle}).
\end{example}

\begin{figure}[h]
  \centering
  \includegraphics[height=5cm]{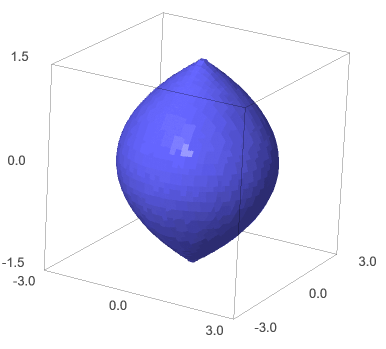}
  \caption{The level set related to the classical limit in
    Example~\ref{ex:1d.quadratic.odd}.}
  \label{fig:spindle}
\end{figure}

\noindent
In the following, we will provide several examples, in which the
generators have the form $A_i=\sqrt{f_n}U^n$ (no summation over $n$),
where $f_n\in\F$ is a \emph{real-valued} function. However, $f_n$ is
not assumed to be strictly positive and we shall use the convention
\begin{align*}
  \sqrt{x}=i \sqrt{|x|}, \, \text{for } x<0.
\end{align*}

\begin{example}[Fuzzy sphere]
\noindent We start with
\begin{align*}
	A_+=\sqrt{R^2-u(u-\hbar)}U^{-1}, \qquad	A_-=\sqrt{R^2-u(u+\hbar)}U
\end{align*}
where $R$ is a real positive number. 
With this definition 
\begin{align*}
	[A_+,A_-] =2\hbar u, \qquad \frac{1}{2} ( A_+A_- + A_-A_+) + 2u^2= 2 R^2,
 \end{align*}
in agreement with the defining relations of the fuzzy sphere.

On the representation $\Shd(\ket{k})$, this becomes
\begin{align*}
  &A_+\ket{n}=\sqrt{R^2-\hbar^2 (n+\delta+1)(n+\delta)}\ket{n+1}\\
  &A_-=\sqrt{R^2-\hbar^2 (n+\delta)(n+\delta-1)}\ket{n-1}.
\end{align*}
The parameters $R$, $\delta$ and $\hbar$ need not be interrelated for infinite
dimensional representations. However, if we demand that there exists
subrepresentations of $\Shd(\ket{k})$ according to Proposition \ref{prop:1d-cases},
then this puts restrictions on the parameters.

In particular, in the notation of Proposition \ref{prop:1d-cases} 
\begin{align*}
	f(u)=\sqrt{R_k^2-\paraa{u+\frac{\hbar}{2}}\paraa{u-\frac{\hbar}{2}}}
\end{align*}
where $k$ is a natural number labeling the representation. 
Requiring $A_+\ket{N_k}=\A_-\ket{M_k}=0$ gives
\begin{align*}
	N_k+\frac{1}{2}+\delta_k =  \pm \sqrt{\frac{R_k^2}{\hbar_k^2}-\frac{1}{4}}, \qquad
	M_k-\frac{1}{2}+\delta_k = \pm \sqrt{\frac{R_k^2}{\hbar^2}-\frac{1}{4}}
\end{align*}
For the choice of different signs of $N_k$ and $M_k$, one can add the two equations resulting in
\begin{align*}
	N_k+M_k+2\delta_k=0
\end{align*}
Since $N_k$ and $M_k$ are integers, $\delta_k$ is either $0$ or $\frac{1}{2}$. Furthermore,
we can set $N_k=k$ and $M_k=-k-2\delta_k$. If we set
\begin{align*}
	S_k =
	\begin{cases}
		 \frac{N_k}{2},  \,	& \text{for } N_k \text{ even} \\
	   	 \frac{N_k+1}{2}, \,  & \text{for } N_k \text{ odd} 
	\end{cases}
\end{align*} 
we arrive at the spin-$S_k$-representations of $su(2)$.
\end{example}

\begin{example}[Fuzzy catenoid]
\noindent  Analogously to the previous example, we start with
\begin{align*}
	A_+=\sqrt{R^2+u(u-\hbar)}U^{-1}, \qquad	A_-=\sqrt{R^2+(u+\hbar)u}U
\end{align*}
where $R>\hbar$ is again a real positive number. Due to the sign change
compared to the previous example, the subalgebra generated by these two
elements fulfills
\begin{align*}
	[A_+,A_-] =- 2\hbar u, \qquad \frac{1}{2} ( A_+A_- + A_-A_+) - u^2= R^2
 \end{align*}
 In the classical limit one can identify $A_+$ and $A_-$ with $x+iy$
 and $x-iy$ and the second equation becomes $x^2+y^2-u^2=R^2$. We can
 think of the subalgebra as a fuzzy version of a catenoid.  The
 representations are infinite dimensional, and one can choose
 arbitrary values for $R$, $\hbar$ and $\delta$.
\end{example}

\begin{example}[Fuzzy cone and fuzzy plane]\label{fuzzy_plane} 
  \noindent In the previous examples, one finds quadratic polynomials
  in the square root.  Here, we investigate a linear polynomial
\begin{align*}
	A_+=\sqrt{u+\hbar c}\,U^{-1}, \qquad	A_-=\sqrt{u+\hbar(c+1)}\,U
\end{align*}
where $c\in\reals$ is a constant. This results in a constant commutator
\begin{align*}
	[A_+,A_-] =- \hbar, \qquad \frac{1}{2} ( A_+A_- + A_-A_+) = u + \hbar (c+\frac{1}{2} )
\end{align*}
In the classical limit the second equation becomes $x^2+y^2=u$. This
is only solvable for $u\geq 0$.  In the representation
$\Shd(\ket{k})$, the action of the generators is
\begin{align*}
	A_+\ket{n}=\sqrt{\hbar(n+\delta+c+1)}\ket{n+1}, \qquad
	A_-=\sqrt{\hbar(n+\delta+c)}\ket{n-1}
\end{align*}
When restricting the representation $\Shd(\ket{k})$ to vectors with
$n\geq N$, where $N$ is an integer, it is necessary that
$\hbar(\delta+c)=N$. The non-integer part of the constant $c$ has to
compensate $\delta$.  Let us assume that $N=0$, since any other $N$
can be reached by redefining $n\rightarrow n-N$.  In the case
$\hbar(\delta+c)=0$, the representation is restricted to $n\geq 0$,
where $\hbar$ can be arbitrary, and the action of the generators
becomes
\begin{align*}
	A_+\ket{n}=\sqrt{\hbar(n+1)}\ket{n+1}, \qquad
	A_-=\sqrt{\hbar n}\ket{n-1}
\end{align*}
The representation restricted to $n\geq 0$ can be seen as a fuzzy version of a cone.
Note that a single cone with an opening angle of 180° is a plane. 
Since 
\begin{align*}
	\rho=\frac{1}{2} ( A_+A_- + A_-A_+) \ket{n}= \hbar(n+\frac{1}{2} ) \ket{n}
\end{align*}
we can identify $\rho$ with $r^2$, where $r$ is the distance to the
origin. In this way, the representation can be viewed as a fuzzy
plane.
\end{example}

\subsection{Shift subalgebras for $D=2$} \label{sec:shift.subalgebras.two.dimensions}

\noindent
In the case when $D=2$, the algebra $\Ahq^2(\F)$ is generated by $U_1,U_2$ together
with functions in an appropriate subalgebra
$\F\subseteq F(\reals^2,\complex)$. We write $u=u_1$ and $v=u_2$ and
consider the shift subalgebra generated by
\begin{align}
 	& U_+=f \left( u-\frac{\hbar}{2},v \right) U_1^{-1}, \qquad
			U_-= f \left( u+\frac{\hbar}{2},v \right) U_1 \label{U_pm_Def} \\
 	& V_+=g \left( u, v-\frac{\hbar}{2} \right) U_2^{-1}, \qquad
			V_-=g \left( u, v+\frac{\hbar}{2} \right) U_2 \label{V_pm_Def}.
\end{align}
for some $f,g\in\F$. In analogy with the one dimensional case, we
denote the shift subalgebra of $\Ahq^2(\F)$ generated by these generators and by 
$f\in\F$ by $\Ahq^2(\F; f, g)$.

The generators of $\Ahq^2(\F; f, g)$ fulfill
  \begin{align*}
	& U_+U_- + U_-U_+ = f^2(u-\frac{\hbar}{2},v)+f^2(u+\frac{\hbar}{2},v)  \\   
	& V_+V_- + V_-V_+ = g^2(u, v-\frac{\hbar}{2})+g^2(u, v+\frac{\hbar}{2}),  
  \end{align*}
  which can be seen as defining
  relations of a fuzzy space. Namely, 
  in the classical limit, when one identifies $U \sim x+iy$ ,
  $V \sim a+ib$, these relations become
\begin{align} 
	x^2+y^2 = f^2(u,v), \qquad a^2+b^2=g^2(u,v)\label{manifold_rel}
\end{align}
Thus, the algebra $\Ahq^2(\F; f, g)$ can be seen as a fuzzy version of a
(generically four dimensional) level set immersed in
$\reals^6$, defined by \eqref{manifold_rel}.

The representation $\Shd(\ket{k})$ of $\Ahq^2(\F)$ has basis vectors
$\ket{n,m}$, where $n,m\in\integers$.  As in the one dimensional case,
it is possible to construct irreducible
representations of the subalgebra $\Ahq^2(\F; f, g)$ by choosing $f$
and $g$ such that their zero loci "cut out" a part of
the $\integers^2$-lattice.  We will show some examples of this in the
following.

\begin{example} 
  \noindent In this example, we shall construct a subalgebra of $\Ahq^2(\F)$
  with finite dimensional representations.  In particular, let $N$ be
  a natural number and let the generators of the subalgebra be defined
  by
\begin{align*}
  & U_+= \sqrt{ \left(u-\hbar \delta_1 \right)  \left(\hbar (\delta_1+\delta_2 + N + 1 )-u-v\right) }
    \tilde{f} \left( u ,v \right) U_1^{-1} \\
  & U_-= \sqrt{ \left(u-\hbar (\delta_1+1) \right)  \left(\hbar (\delta_1+\delta_2 + N )-u-v\right) }
    \tilde{f} \left( u ,v \right) U_1 \\
  & V_+= \sqrt{ \left(v- \hbar \delta_2 \right)  \left(\hbar (\delta_1+\delta_2 + N + 1 )-u-v \right) }
    \tilde{g} \left( u, v \right) U_2^{-1} \\
  & V_-= \sqrt{ \left(v- \hbar \delta_2 \right)  \left(\hbar (\delta_1+\delta_2 + N  )-u-v \right) }
    \tilde{g} \left( u, v \right) U_2,
\end{align*}
where $\tilde{f},\tilde{g}\in\F$ denote arbitrary \emph{positive} functions.
The action of the above generators is given by
\begin{align*}
 	& U_+\ket{n,m} = \hbar \sqrt{(n+1)(N-n-m)}\
			\tilde{f} \paraa{ \hbar(n+1+\delta_1),\hbar(m+\delta_2) } \ket{n+1,m} \\
 	&  U_-\ket{n,m} = \hbar \sqrt{n (N-1 -n-m)}\
			\tilde{f} \paraa{ \hbar(n+\delta_1),\hbar(m+\delta_2) } \ket{n-1,m} \\
 	& V_+\ket{n,m} =  \hbar \sqrt{(m+1)(N-n+m)}\
			\tilde{g} \paraa{\hbar(n+\delta_1),\hbar(m+1+\delta_2) } \ket{n,m+1}  \\
 	 & V_-\ket{n,m} =  \hbar \sqrt{m(N-1-n-m)}\
			\tilde{g} \paraa{ (\hbar(n+\delta_1),\hbar(m+\delta_2) } \ket{n,m-1}.
\end{align*}
The important observations are now that $U_-\ket{0,m}=0$,
$V_-\ket{n,0}=0$ and that
$U_+\ket{\hat{n},\hat{m}}=V_+\ket{\hat{n},\hat{m}}=0$ for
$\hat{n}+\hat{m}=N$.  One sees that when one starts with the
generating vector $\ket{0, 0}$ and applies a monomial in $U_+$, $V_+$,
$U_-$ and $V_-$, it is not possible to leave the part of the lattice
defined by $n \geq 0$, $m \geq 0$ and $n+m \leq N$.  Thus, the
irreducible representation generated from $\ket{0, 0}$ is finite
dimensional and has dimension $\frac{1}{2}(N+1)(N+2)$. For the
classical limit, where $N\rightarrow\infty$, one assumes that
$\hbar N=C$ for some $C\in\reals$, i.e. that when $N$ increases then
$\hbar$ decreases. This is analogous to the fuzzy sphere, where the
assumption that the radius should be constant also interrelates
$\hbar$ with the size of the representation $N$.

In the classical limit, the algebra relations give two equations, which define a
four dimensional level set in $\reals^6$ as follows 
\begin{align*}
	x^2+y^2= u(C-u-v) \tilde{f}^2, \qquad a^2+b^2=v(C-u-v) \tilde{g}^2.
\end{align*}
For instance, when $\tilde{f}=1$, a section through the level set at
constant $v$ is a sphere with radius $(C-v)/2$.
\end{example}

\begin{example}
\noindent Let us now construct a subalgebra with infinite dimensional 
representations. The generators of the subalgebra of $\Ahq^2(\F)$ are
\begin{align*}
 	& U_+= \sqrt{ \left(u-\hbar \delta_1 \right) \left(v- \hbar \delta_2 \right) }
						\tilde{f} \left( u ,v \right) U_1^{-1} \\
 	& U_-= \sqrt{ \left(u-\hbar(1+ \delta_1) \right) \left(v- \hbar \delta_2 \right) }
						\tilde{f} \left( u ,v \right) U_1 \\
 	& V_+= \sqrt{ v-\hbar \delta_2 }
						\tilde{g} \left( u, v \right) U_2^{-1} \\
 	& V_-= \sqrt{ (v-\hbar (\delta_2+1) }
						\tilde{g} \left( u, v \right) U_2
\end{align*}
Again, $\tilde{f},\tilde{g}\in\F$ denote arbitrary positive functions.
We will now show that the irreducible representation generated by
$\ket{0,0}$ is restricted to $\{\ket{n,m}:n\geq 0, \, m\geq 0 \}$.  To
this end, let us calculate the action of the generators
\begin{align*}
 	& U_+\ket{n,m} = \hbar \sqrt{(n+1)m}\
			\tilde{f} \paraa{ \hbar(n+1+\delta_1),\hbar(m+\delta_2) } \ket{n+1,m} \\
 	&  U_-\ket{n,m} = \hbar \sqrt{n m}\
			\tilde{f} \paraa{ \hbar(n+\delta_1),\hbar(m+\delta_2) } \ket{n-1,m} \\
 	& V_+\ket{n,m} =  \sqrt{\hbar(m+1)}\
			\tilde{g} \paraa{\hbar(n+\delta_1),\hbar(m+1+\delta_2) } \ket{n,m+1}  \\
 	 & V_-\ket{n,m} =  \sqrt{\hbar m}\
			\tilde{g} \paraa{ (\hbar(n+\delta_1),\hbar(m+\delta_2) } \ket{n,m-1}.
\end{align*}
It follows that $U_-\ket{0,m}=0$ and that for $n>0$,
$U_-\ket{n,m}\neq 0$ . Vice versa, for $n\geq 0$,
$U_+\ket{n,m}\neq 0$.  In the same way, $V_-\ket{n,0}=0$ and for $m>0$
$V_-\ket{n,m}\neq 0$. Vice versa, for $m\geq 0$ $V_+\ket{n,m}\neq 0$.
Starting from $\ket{0,0}$, every vector $\ket{n,m}$ for $n\geq 0$ and
$m\geq 0$ is reachable by applying $U_+$ and $V_+$. On the other hand,
the vectors $\ket{n,m}$ for $n<0$ and $m<0$ are not reachable from
$\ket{0,0}$, since it is not possible to move outside the lines $n=0$
and $m=0$ in $\integers^2$. Therefore, the representation
$\Shd(\ket{k})$ is infinite dimensional and has basis vectors
$\ket{n,m}$ with $n,m\geq0$.

In the classical limit the algebra relations become
\begin{align*}
x^2+y^2=uv\tilde{f}(u,v), \qquad a^2+b^2=v\tilde{g}(u,v),
\end{align*}
defining a non-compact level set in $\reals^6$.
Note that the restriction to $n\geq 0$ and $m\geq 0$ in
the representation restricts the coordinates in the classical limit to
$u\geq 0$ and $v\geq 0$.  This example is particularly interesting,
since if one sets $u=r+z$, $v=r-z$ and $\tilde{f}=1$, then, in the
classical limit, the algebra relations become
\begin{align*}
x^2+y^2+z^2=r^2, \qquad a^2+b^2=(r-z)\tilde{g}(r,z).
\end{align*}
Since $r=\frac{1}{2}(u+v)$, one may interpret $r$ as a radius since it
is non-negative because of $u\geq 0$ and $v\geq 0$. Additionally, we
can parametrize the coordinates $a$ and $b$ with a compact parameter
$\tau$:
\begin{align*}
  a=\sqrt{(r-z)\tilde{g}(r,z)}\cos\tau, \qquad
  b=\sqrt{(r-z)\tilde{g}(r,z)}\sin\tau
\end{align*}
Thus, the level set can be parametrized by $\reals^3\times S^1$, where
$S^1$ is a circle with radius $\sqrt{(r-z)\tilde{g}(r,z)}$, vanishing
when $r=z$.
\end{example}

\section{Representations of Lie algebras} \label{sec:shift.subalgebras.higher.dimensions}

\noindent
In this section we show that, for any finite dimensional Lie algebra
$\g$, one can define a shift subalgebra where the generators satisfy
the commutation relations of $\g$. Consequently, one may construct
irreducible representations $\Shd(\ket{k})$ of $\g$ by using
Proposition~\ref{prop:AhqDFan.simple.module}.  Moreover, the Lie
algebra generators in the shift subalgebra provide a natural set of
inner derivations, suggesting the interpretation of the shift
subalgebra as a ``noncommutative homogeneous space''.  A lot of
literature exists, in which such fuzzy homogeneous spaces are
constructed. For instance, it is possible to use coherent states
\cite{GP1993}, highest weight representations of Lie algebras
\cite{Alexanian_2002}, \cite{Balachandran_2002} or subspaces of Fock
space representations \cite{Grosse_1996}, \cite{Grosse_1999}. In this
section, we illustrate another way of arriving at such fuzzy
homogeneous spaces.

We start with a shift algebra $\A_{\hbar,q}^D(\F)$, and assume in this
section that $q=1$. For notational convenience, we write
$S_i=S_{\hbar,i}$ (cf. Section~\ref{sec:shift.algebras}) such that
\begin{align*}
  (S_if)(u_1,\ldots,u_D) = f(u_1,\ldots,u_i+\hbar,\ldots,u_D).
\end{align*}

\begin{lemma} \label{deformed_suD} Let $\A_{\hbar,1}^D$ be a shift
  algebra and let $f_i\in\F$, for $i=1,\ldots,D$,
  such that $S_kf_i=f_i$ for $k\neq i$. Then the shift subalgebra
  generated by 
  \begin{align*}
    E_{ij}=f_i U_i^{-1} U_j f_j\qquad\quad (i,j=1,\ldots,D)
  \end{align*}
  (no summation over $i$ and $j$) fulfills the commutator relations
  \begin{align} \label{Eij_comm}
    [E_{ij},E_{kl}]=\delta_{jk}\Delta_j E_{il}-\delta_{il} \Delta_i E_{kj}
  \end{align}
  where $\Delta_i=S_i ( f_i ^2 )  - f_i^2$. 	
\end{lemma}

\begin{proof}
By direct calculation one shows that for $i\neq j \neq k \neq l$
\begin{align*}
	& E_{ij} E_{kl} =E_{kl} E_{ij}   \implies [ E_{ij}, E_{kl} ]=0 \\
	&  E_{ii} E_{jj} = E_{jj} E_{ii} \implies [E_{ii}, E_{jj}]=0 \\
	& E_{ij} E_{kj} =E_{kj} E_{ij} \implies [E_{ij}, E_{kj}]=0 \\
	& E_{ij} E_{ik} =E_{ik} E_{ij}  \implies [E_{ij}, E_{ik}]=0 \\
	& E_{ij} E_{jk} =  S_j ( f_j^2) E_{ik}, \qquad E_{jk}E_{ij}  =  f_j^2 E_{ik}
		  \implies  [ E_{ij}, E_{jk} ] =  \left( S_j ( f_j^2)  - f_j^2 \right)  E_{ik}\\
	&  E_{ij} E_{jj} =  S_j ( f_j^2) E_{ij}, \qquad   E_{jj} E_{ij} =  f_j^2 E_{ij} 
		 \implies [E_{ij}, E_{jj}] = \left( S_j ( f_j^2 )  - f_j^2 \right)  E_{ij} \\
	& E_{ij} E_{ji} =    f_i^2 \, S_j ( f_j^2 ) = S_j ( f_j^2 )E_{ii} \\
	& \implies [E_{ij}, E_{ji}] = \left( S_j ( f_j^2)  - f_j^2 \right)  E_{ii}
							- \left( S_i ( f_i^2 )  - f_i^2 \right)  E_{jj}
\end{align*}	
These relations can be summarized into (\ref{Eij_comm}).
\end{proof}

\noindent
If $\Delta_i$ is constant and independent of $i$, then in Lemma \ref{deformed_suD},
the commutator relations (\ref{Eij_comm})
reduce to the Lie algebra relations of $u(D)$.  In this case, the shift subalgebra generated by the
$E_{ij}$ is an enveloping algebra of $u(D)$ (but not the universal one).
Setting $\Delta_i=\hbar$ can be achieved by choosing $f_i=\sqrt{u_i+c}$, for
arbitrary $c\in\reals$. In the following we show that this can be
generalized to any finite dimensional Lie algebra $\g$.

\begin{proposition}\label{lie_alg}
  Let $\A_{\hbar,1}^D(\F)$ be a shift algebra such that
  $\sqrt{u_i+c_i}\in\F$ for $c_i\in\reals$ and $i=1,\ldots,D$, and let
  $\g$ denote a finite dimensional Lie algebra with basis
  $\{X^a\}_{a=1}^n$ such that $[X^a, X^b]= f \indices{^{ab}_c} X^c$.
  Furthermore, let $(X^a_{ij})$ denote the matrices
  of a $N$-dimensional representation of $\g$.
  Then the elements
  \begin{align*}
    \hat{X}^a = \sum_{i,j=1}^N X^a_{ij} \sqrt{u_i+c_i} U_i^{-1} U_j \sqrt{u_j+c_j}
  \end{align*}
  satisfy $[\hat{X}^a,\hat{X}^b]=\hbar {f^{ab}}_c\hat{X}^c$ for $a,b=1,\ldots,n$.
\end{proposition}

\begin{proof}
  In the notation of Lemma~\ref{deformed_suD}, one finds
  that $\Delta_i=u_i+\hbar+c_i-u_i-c_i=\hbar$. Defining
  \begin{align} \label{suD_gen}
     E_{ij}=\sqrt{u_i+c_i} U_i^{-1} U_j \sqrt{u_j+c_j}
  \end{align}
  and setting
  \begin{align*}
    \hat{X}^a = \sum_{i,j=1}^N X^a_{ij} E_{ij},
  \end{align*}
  it follows immediately from (\ref{Eij_comm}) that
  \begin{equation*}
    [\hat{X}^a, \hat{X}^b] = \hbar f \indices{^{ab}_c} \hat{X}^c. \qedhere
  \end{equation*}
\end{proof}

\noindent
Thus, Proposition~\ref{lie_alg} implies that the shift subalgebra
generated by $\{\hat{X}^a\}_{a=1}^n$ is an enveloping algebra of the Lie algebra $\g$.

Let us now present the main result of this section, showing that one
can construct irreducible representations of an arbitrary finite
dimensional Lie algebra by considering representations of shift
subalgebras.

\begin{proposition}\label{lie_alg_rep}
  Let $\g$ be a $n$-dimensional Lie algebra, and let $\A$ be the shift
  subalgebra generated by $\{\hat{X}^a\}_{a=1}^n$, as defined in
  Proposition~\ref{lie_alg}. For $n=(n_1,\ldots,n_D)\in\integers^D$,
  the simple $\A$-module $\Shd(\ket{n})$
  (cf. Proposition~\ref{prop:AhqDFan.simple.module}) is an irreducible
  representation of $\g$. Moreover, 
  \begin{align*}
    \Shd(\ket{n_1, \ldots, n_D}) \subseteq
    \operatorname{span}_{\complex}\{\ket{k_1, \ldots, k_D}: k_1+ \cdots +k_D=n_1+ \cdots +n_D\}.
  \end{align*}
  If $\hbar\delta_i+c_i = 0$ and $n_i\geq 0$ for $i=1,\ldots,D$ then
  $\Shd(\ket{n_1, \ldots, n_D})$ is finite dimensional.
\end{proposition}

\begin{proof}
  We know from Proposition 4.2 that $\Shd(\ket{n_1, \ldots, n_D})$ is
  a simple $\A$-module. Moreover, there it is
  shown that every vector is cyclic, which is also the case for the
  $\hat{X}^a$, which represent the Lie algebra basis
  elements. Therefore, $\Shd(\ket{n_1, \ldots, n_D})$ is also simple
  (or irreducible) as a representation of the Lie algebra $\g$.

  Furthermore, the sum $k_1+k_2+\cdots+ k_D$ is constant under the
  action of the $E_{ij}$ of (\ref{suD_gen}). Namely, for $i\neq j$ it follows that
  \begin{align*}
    & E_{ij} \ket{k_1, \ldots, k_D} = \sqrt { \hbar(k_i+\delta_i +1)+c_i)(\hbar(k_j+\delta_j)+c_j) }
			 \ket{k_1, \ldots, k_i+1, \ldots, k_j-1, \ldots, k_D}  \\
    & E_{ii} \ket{k_1, \ldots, k_D} =  \paraa{ \hbar(k_i+\delta_i)+c_i } \ket{k_1, \ldots, k_D}  
  \end{align*}
  Consequently, the sum $k_1+\cdots+k_D$ is also constant under the
  action of an arbitrary element in the shift subalgebra, since the
  shift subalgebra is generated by the $\hat{X}^a$, which are linear
  combinations of the $E_{ij}$.

  Assuming that $\hbar\delta_i+c_i=0$,  the action
  of  the $E_{ij}$ becomes
  \begin{align*}
    & E_{ij} \ket{k_1,\dots,k_D}= 
      \hbar \sqrt{\paraa{k_i+1} k_j } \ket{k_1,\dots,k_i+1,\dots,k_j-1,\dots, k_D}  \\
    & E_{ii} \ket{k_1,\dots,k_D}= 
      \hbar k_i \ket{k_1,\dots,k_i,\dots, k_D}
  \end{align*}
  From this follows that if $n_i\geq 0$, the module
  $\Shd(\ket{n_1, \ldots, n_D})$ is restricted to a subset of the
  basis elements $\ket{k_1,\ldots, k_D}$ with $k_i\geq 0$, since the
  factor $\sqrt{\paraa{k_i+1} k_j }$ for the operators $E_{ij}$ with
  $i\neq j$ becomes $0$, when applied to a vector with $k_j=0$.  In
  the case $\hbar\delta_i+c_i=0$ and $n_i\geq 0$ this means that
  $\Shd(\ket{n_1, \ldots, n_D})$ is a subspace of the vector space
  spanned by $\ket{k_1, \ldots, k_D}$ where
  $k_1+ \cdots +k_D=n_1+ \cdots +n_D$ and $k_i\geq 0$. These
  $\{k_1, \ldots, k_D\}$ form a finite subset of $\integers^D$.

  The shift subalgebra $\A$ generated by the $\hat{X}^a$ is a subalgebra of
  the shift subalgebra generated by the $E_{ij}$ , since the $\hat{X}^a$ are linear combinations
  of the $E_{ij}$.  Thus, because the module $\Shd(\ket{n_1, \ldots, n_D})$ 
  is finite dimensional for the shift subalgebra generated by the $E_{ij}$, this is
  also the case for $\A$.
\end{proof}

\noindent
It is an interesting observation that the hyperplane in $\integers^D$, which is defined 
by $k_1+ \cdots +k_D=N$ with $N$ some integer forms a triangular lattice.
The condition $k_i\geq 0$ cuts out a regular $D$-simplex from this hyperplane.

Let us end with two examples as an illustration of the above results.

\begin{example}[$su(2) \simeq so(3)$] 
  We explore the simplest non-trivial example, the Lie algebra $su(2)$ and its
  fundamental representation, which is provided by the Pauli matrices
  \begin{align*}
    \sigma_{1}=\left(\begin{array}{cc} 0 & 1\\ 1 & 0 \end{array}\right),
                                                   \quad
                                                   \sigma_{2}=\left(\begin{array}{cc} 0 & -i\\ i & 0 \end{array}\right),\quad
                                                                                                   \sigma_{3}=\left(\begin{array}{cc} 1 & 0\\ 0 & -1 \end{array}\right)
  \end{align*}
  This is a two-dimensional representation, implying that the subalgebra
  defined in Proposition~\ref{lie_alg} is a shift subalgebra of
  $\A_{h,1}^2(\F)$. The generators are given by
  \begin{align*}
    & X_{1}= \sqrt{u_1+c_1} U_1^{-1} U_2 \sqrt{u_2+c_2} + \sqrt{u_2+c_2} U_2^{-1} U_1 \sqrt{u_1+c_1}  \\
    & X_{2}=-i\sqrt{u_1+c_1} U_1^{-1} U_2 \sqrt{u_2+c_2} + i\sqrt{u_2+c_2} U_2^{-1} U_1 \sqrt{u_1+c_1} \\
    & X_{3}=u_1+c_1-u_2-c_2
  \end{align*}
  We know from Proposition \ref{lie_alg_rep} that the representation $\Shd(\ket{n_1, n_2})$ is restricted at least to the vectors
  $\ket{k, N-k}$ with $N=n_1+n_2$. The operators become
  \begin{align*}
    & A_+\ket{k, N-k} = \sqrt{\paraa{\hbar(k+1+\delta_1)+c_1}{\hbar(N-k+\delta_2)+c_2}} \ket{k+1, N-k-1} \\
    & A_-  \ket{k, N-k} = \sqrt{  \paraa{\hbar(N-k+1+\delta_2)+c_2} \paraa{\hbar(k+\delta_1)+c_1}} \ket{k-1, N-k+1} \\
    & Z \ket{k, N-k} = \paraa{\hbar(k+\delta_1)+c_1-\hbar(N-k+\delta_2)-c_2} \ket{k, N-k}
  \end{align*}
where we have introduced the ladder operators
  \begin{align*}
     A_+=\frac{1}{2}(X+iY), \quad A_- = \frac{1}{2i}(X-iY), \quad Z=X_{3}
  \end{align*}
  Depending on $k$, the expressions below the square roots can be positive and negative.
  As shown in Proposition \ref{lie_alg_rep} we see that a restricted representation
  is only possible, when the expression below the square root
  can become $0$ for some $k$, which is only possible, when the $c_i$ compensate
  the $\delta_i$. Let us therefore assume that $\hbar\delta_i+c_i=0$. We then arrive at
  \begin{align*}
    & A_+\ket{k, N-k} = \hbar \sqrt{\paraa{k+1}\paraa{N-k}} \ket{k+1, N-k-1} \\
    & A_- \ket{k, N-k} = \hbar \sqrt{ \paraa{N-k+1} k} \ket{k-1, N-k+1} \\
    & Z \ket{k, N-k}=\hbar(2k-N) \ket{k, N-k}
  \end{align*}
  The expressions below the square roots are positive for $0<k<N$ otherwise they are negative.
  Furthermore, $A_+ \ket{N, 0}=0$ and $A_- \ket{0, N}=0$. The irreducible representations
  are the representations $\Shd(\ket{N, 0})$ and are up to a shift the representations
  of the fuzzy sphere.
\end{example}

\begin{example}[$su(D)$]
  To generalize the previous example, for $su(D)$ we can directly use
  the algebra elements $E_{ij}$ from Proposition \ref{lie_alg_rep}.
  From this proposition we know that the representation
  $\Shd(\ket{n_1, \ldots, n_D})$ can be restricted to basis vectors
  $\ket{k_1, \ldots, k_D}$ with $k_1+\dots+k_D=n_1+\dots+n_D=N$. Since
  for $k_i>0$, the factors under the square root of the $E_{ij}$ are
  all greater than $0$, the representation sweeps out
  the complete set $k_1+\dots+k_D=N$ restricted to $k_i>0$. For the
  same $N$ the representations are equivalent. For every $N$,
  one obtains an irreducible representation of $su(D)$.

  The operators $E_{ii}$ in the representation are linearly dependent. They
  correspond to the diagonal generator $Z$ of the previous example. 
  Since the factors under the square root of the operators $E_{ij}$ with $i\neq j$
  are all non-negative for all $k_1,\dots, k_D$ with $k_i \geq 0$ it follows that
  the operators $E_{ij}$ are conjugate to the operators $E_{ji}$.
  These operators correspond to the conjugate generators $A_+$ and $A_-$
  of the previous example. 

  The representations generated in this way are usually denoted by
  $[N, 0, \dots, 0]$, where each of the numbers in the square brackets
  indicates a highest weight vector of $su(D)$. It is known that these
  representations are fuzzy complex projective spaces;
  see e.g. \cite{Alexanian_2002,Balachandran_2002}.
\end{example}

\section*{Acknowledgments}

\noindent
J.A. is supported by the Swedish Research Council grant 2017-03710.

\bibliographystyle{alpha}
\bibliography{references}

\end{document}